\documentclass[11pt]{article}
\usepackage{setspace}
\usepackage{mathrsfs}
\usepackage{graphicx}
\usepackage{multirow}
\usepackage{comment}
\usepackage[all]{xy}
\usepackage[usenames]{color}
\usepackage[noadjust]{cite}
\usepackage{amsfonts}
\usepackage{amssymb}
\usepackage{amsmath}
\usepackage{amsthm}
\usepackage{algorithm}
\usepackage{algorithmic}
\usepackage{enumerate}
\usepackage{tabularx}
\usepackage{booktabs}
\usepackage[labelfont=bf,format=plain,justification=raggedright,singlelinecheck=false]{caption}
\usepackage[misc]{ifsym}
\usepackage[unicode=true,pdfusetitle,
bookmarks=true,bookmarksnumbered=false,bookmarksopen=false,
breaklinks=false,pdfborder={0 0 1},colorlinks=false]
{hyperref}
\theoremstyle{plain}
\newtheorem{thm}{Theorem}[section]

\newtheorem{lem}[thm]{Lemma}
\newtheorem{prop}[thm]{Proposition}

\newtheorem{claim}[thm]{Claim}

\theoremstyle{definition}
\newtheorem{defn}[thm]{Definition}
\newtheorem{rem}[thm]{Remark}

\DeclareMathOperator{\Lip}{Lip}

\DeclareMathOperator{\OrbLip}{OrbLip}

\DeclareMathOperator{\Fix}{Fix}
\DeclareMathOperator{\Int}{int}
\DeclareMathOperator{\Xnf}{X_{nf}}
\DeclareMathOperator{\Edge}{\partial}
\DeclareMathOperator{\spann}{span}

\DeclareMathOperator{\diam}{diam}
\DeclareMathOperator{\shift}{shift}

\DeclareMathOperator{\con}{con}

\DeclareMathOperator{\equi}{equivariant}
\DeclareMathOperator{\supp}{supp}
\def\R{\mathbb{R}}
\def\Rk{\mathbb{R}^k}
\def\N{\mathbb{N}}
\def\Z{\mathbb{Z}}

\def\tt{\mathbf t}

\def\ss{\mathbf s}
\def\rr{\mathbf r}
\def\yy{\mathbf y}
\def\aa{\mathbf a}
\def\ll{\mathbf \lambda}
\def\ww{\mathbf w}
\def\AA{\mathbf A}
\def\DD{\mathbf D}
\def\LL{\mathbf \Lambda}
\def\bb{\mathbf b}

\def\uu{\mathbf u}
\def\vv{\mathbf v}
\def\00{\mathbf 0}
\usepackage{tikz}
\usetikzlibrary{positioning,calc,decorations.pathreplacing}

\usepackage[markup=underlined,commandnameprefix=always]{changes}

\definechangesauthor[color=red]{QH}

\numberwithin{figure}{section}

\begin{document}
\title{{ \Large\bf{On the representation of measurable and  continuous  dynamical systems by Lipschitz functions}}
}
\author{Yonatan Gutman,\,\,\,\,  Qiang Huo
 }
\date{\today}
 \maketitle
\begin{minipage}{14cm} 

{\bf Abstract:}
Two representations theorems are presented:
\begin{enumerate}
    \item Any Borel action  of a second countable locally compact group  $G$ on a standard Borel space $X$  admits an injective $G$-equivariant Borel map into the shift space of $1$-Lipschitz functions from $G$ to the unit interval $\Lip_1(G)$. 
    \item 
    Any continuous action of $\R^k$ ($k\in \N$) on a metrizable compact space $X$  admits an injective $G$-equivariant continuous map into $\Lip_1(\R^k)$ if the fixed point set $\Fix(X,\R^k)$ embeds into $[0,1]$ and $(X,\mathbb{R}^k)$ is \textit{weakly locally free}, that is $\mathbb{R}^k$ acts freely outside the  
    fixed point set.
\end{enumerate}
 The first theorem generalizes a theorem from 1973 by Eberlein for $\mathbb{R}$-flows. The second theorem generalizes  a Lipschitz refinement of the Bebutov--Kakutani theorem proven by Gutman, Jin and Tsukamoto in 2019. \\

 \noindent{\it Keywords:} Topological models; Borel systems;
  Eberlein's theorem; Real flows; Multidimensional flows; Second countable locally compact groups; Lipschitz functions. \\
\noindent{\it 2020 AMS Subject Classification: 37A05; 37B05; 54H05; 54H20.}
\end{minipage}
 \maketitle
\numberwithin{equation}{section}
\newtheorem{theorem}{Theorem}[section]
\newtheorem{lemma}[theorem]{Lemma}
\newtheorem{proposition}[theorem]{Proposition}
\newtheorem{corollary}[theorem]{Corollary}

\tableofcontents

\section{Introduction}

Consider $\Lip_{1}(\Rk)$, the space of $1$-Lipschitz functions (w.r.t. the Euclidean norms) from $\Rk$ to the unit interval , equipped with the compact-open topology, for $k\geq 1$. It admits the natural $\Rk$-shift action $\tt\cdot g( \ss)=g(\ss+\tt)$ $(\tt,\mathbf s \in \R^k)$. In 1973 Eberlein proved the following remarkable theorem:

\begin{thm}(\cite[Theorem 3.1]{Ebe73})
\label{Eberlein}\footnote{See also \cite[Section 2.4]{EFKKS17}. The article \cite{Ebe73} is based on Eberlein's Ph.D thesis and is in German. The notions used in the statement of the theorem are defined in the Preliminaries section.}
The topological flow $(\Lip_{1}(\R),\shift)$ is a topological model for all free measure-preserving $\mathbb{R}$-systems.
\end{thm}
\noindent
Let $G$ be a second countable locally compact group and fix a right-invariant compatible metric $d$ on $G$ (such a metric exists by \cite{struble1974metrics}). Similarly to the above we may consider $\Lip_{1}(G)$, the (compact) space of $1$-Lipschitz functions from $G$ to the unit interval (w.r.t. $d$), equipped with the compact-open topology. The group $G$ acts on $\Lip_1(G)$ via the right argument shift, $(gf)(h)=f(hg)$.
We present a far reaching generalization of Eberlein's theorem:

\begin{thm}\label{thm:LCSC_Eberlein}
  Let $G$ be a second countable locally compact group.  Let $(X,\mathcal{B}, G)$ be a Borel system. Then there exists an injective $G$-equivariant Borel map $\Phi:X\to\Lip_1(G)$.
\end{thm}
Previously  Jahel and Zucker \cite[Theorem 12]{JZ21} had established  that for a second countable locally compact group $G$, any (free) Borel system $(X,\mathcal{B}, G)$ admits an injective $G$-equivariant Borel map into $(\Lip_1(G))^{\N}$. 

We now ask the question to what extent an analogous statement to Theorem \ref{thm:LCSC_Eberlein} holds for topological dynamical systems. 
Specifically, let $(G,X)$ be a topological dynamical system. One may ask under which conditions $(G,X)$ admits an equivariant embedding into $\Lip_1(G)$. This problem is open already in the case $G=\Z$. Indeed, for $G=\Z$ one has $\Lip_1(\Z) = [0,1]^\Z$, since for any $x=(x_n)_{n\in \Z}\in [0,1]^\Z$ and any $n\neq m$, $|x_n-x_m|\leq 1 \leq |n-m|$. Consequently, in the case $G=\Z$, the question above is at least as difficult as the following classical embedding problem in topological dynamics: to determine under which conditions a topological dynamical system $(X,\Z)$ admits an equivariant embedding into $[0,1]^\Z$ (see, for instance, \cite{LT12,GQS18,gutman2020embedding,dranishnikov2025free}). 
In contrast, the case of $G=\R$ has been solved. In \cite{GJT19} the following  was proven:

\begin{thm}\label{GJT}(\cite[Theorem 1.3]{GJT19}) A topological $\mathbb{R}$-flow $(X,\R)$ embeds into $\Lip_{1}(\mathbb{R})$ if and only if $\Fix(X,\R)$ embeds into the unit interval $[0,1]$.
\end{thm}
\noindent
One should note this is a  refinement of the classical Bebutov--Kakutani theorem \cite{Beb40, Kak68}  which states any topological $\mathbb{R}$-flow whose fixed points embed into the unit interval, may be embedded into the space of continuous functions $C(\mathbb{R}, [0,1])$.\footnote{Jaworski generalized the Bebutov--Kakutani theorem to the setting of connected locally compact group actions, in particular actions by $\Rk$ (\cite[Corollary III.1.1]{Jawo74}).} Going beyond the $\R$-case we now state a theorem for $G=\R^k$: 

\begin{defn}\label{def:LFOF}
 A topological  multidimensional flow $(X,\mathbb{R}^k)$ is said to be \textbf{weakly\footnote{The word \textit{weakly} reminds us that such a flow may have fixed points.} locally free}  if for all $x\in X\setminus\Fix(X,\mathbb{R}^k)$ there exists an open set $\mathbf{0}\in U_x\subset \mathbb{R}^k$ such that for all $\mathbf u\in U_x\setminus \{\mathbf{0}\}$ it holds $\uu x\neq x$. 
\end{defn}
\begin{thm}\label{RkGJT}
   Let $k\in \N$. A topological multidimensional flow $(X,\mathbb{R}^k)$ embeds into $\Lip_{1}(\mathbb{R}^k)$ if $\Fix(X,\R^k)$ embeds into $[0,1]$ and $(X,\mathbb{R}^k)$ is weakly locally free.
 \end{thm}
The multidimensional nature of Theorem \ref{RkGJT}  necessitates the application of both classical and recent results on the extension of Lipschitz functions. In view of the discussion above we leave as an open problem to what extent one may generalize Theorem \ref{RkGJT} to other groups beyond $\Rk$.

Eberlein's theorem  played a pivotal role in Denker and Eberlein's result in \cite{DE74} that every ergodic measurable flow is isomorphic to a strictly ergodic subflow of $(\Lip_{1}(\mathbb{R}),\shift)$. This is an $\R$-version of the celebrated Jewett--Krieger theorem (\cite{Jew69, Kri72})\footnote{An instructive proof following \cite{w85}  is given in \cite[\S15.8]{G03} The strictly ergodic model is achieved as an inverse limit of symbolic subshifts.}. Nowadays the  Jewett--Krieger theorem is known in the generality of free ergodic countable amenable group actions (\cite{w85, Ros85,Wei25}), however it is unknown if it holds in the generality of  free ergodic measurable $\Rk$-flows. This is the subject of a forthcoming work of the authors. 

\begin{rem}\label{notation_Lipfirst}
     Note that the conditions appearing in Theorem \ref{RkGJT}  are sufficient for embedding but not necessary. This is unlike the situation in Theorem \ref{GJT} where a sufficient and necessary condition for embedding is given. In a subsequent independent work, the authors, together with Masaki Tsukamoto, were able to remove the weakly locally free condition and prove that a topological multidimensional flow $(X,\mathbb{R}^k)$ embeds into $\Lip_{1}(\mathbb{R}^k)$ if $\Fix(X,\R^k)$ embeds into $[0,1]$ (\cite[Theorem 1.7]{gutman2025lipschitz}).  While the proof of the more general result ultimately relies on a highly technical perturbation argument—much as in the present paper—the two approaches differ substantially in the methods used to implement the perturbation (cf. Lemma 4.7 in the present work and \cite[Propositions 3.1 \& 3.3]{gutman2025lipschitz}). This distinction may prove advantageous for future developments.
     \end{rem}

\section*{Acknowledgments}
\indent\indent  
This work began while the second-named author visited the Institute of Mathematics of the Polish Academy of Sciences in Warsaw. He would like to thank the institute for its hospitality. Y.G. was partially supported by the National Science Centre (Poland) grant 2020/39/B/ST1/02329. Q.H. was partially supported by the China Scholarship Council CSC - 202206040090; China Postdoctoral Science Foundation 2025M773065; National Key Research and Development Program of China 2024YFA1013600; Fundamental Research Funds for the Central Universities WK0010250102. The authors would like to thank Lei Jin, Konstantin Slutsky, Masaki Tsukamoto and Benjamin Weiss for useful comments and suggestions.

\section{Preliminaries}\label{sec:Preliminaries}
\subsection{General notation and standing conventions}
We use boldface to denote vectors in $\Rk$ and subscripts to denote the coordinates of a vector. 
We use $\|\cdot\|_1,\|\cdot\|_{\infty}$, and $\|\cdot\|_2$ for the $\ell^1,\ell^\infty$ and Euclidean norms in Euclidean spaces respectively.
Given a metric space $(M,d)$, for $x\in M$ and $r>0$, denote by $B_r(x)$ respectively $\overline{B}_r(x)$ the open respectively closed $r$-ball around $x$. Given two sets $A,B\subset M$, write $d(A,B)=\inf_{a\in A,b\in B} d(a,b)$. 
\subsection{Actions}
Let $H$ be a group and $Y$ a set. The unit element of $H$ is denoted by $e$. An \textbf{action} is a map $\Psi:H\times Y\rightarrow Y$, such that $\Psi(e,x)=x$ and $\Psi(gh,x)=\Psi(g,\Psi(h,x))$ for all $x\in Y$, $g,h \in H$. We often denote $\Psi(g,x):= gx$. Similarly  for $A\subset Y$ we denote $\Psi(g,A):=gA$ and for $J\subset H$ we denote $\Psi(J,A):=J\cdot A$. 

\subsection{Borel systems}
 A \textbf{standard Borel space} $(X,\mathcal{B})$
is a measure space isomorphic to a Polish\footnote{A topological space is called \textbf{Polish} if its topology can be metrized by a complete and separable metric.} space  together with its $\sigma$-algebra of Borel subsets.
A \textbf{Borel system} is a triple  $(X,\mathcal{B}, G)$, where $(X,\mathcal{B})$ is a  standard Borel space and $G$ is a second countable locally compact group, equipped with an 
action $\Psi:G\times X\rightarrow X$, where $\Psi$  is a Borel map (\cite[\S 2.1]{BK96}).  A Borel system $(X,\mathcal{B}, G)$ is said to
be \textbf{free} if $gx=x$ for some $x\in X$ implies $g=e$.

Note that a Borel system  is not assumed to be measure-preserving with respect to some measure.

\begin{rem}
    In this article all measure spaces are standard Borel spaces and measurable means Borel measurable. Recall that the composition of Borel maps is Borel. A Borel subspace of a  Borel space is a standard Borel space. 
\end{rem}

\subsection{Measure preserving systems}\label{subsec:Measure preserving systems}

A \textbf{measure preserving system} (m.p.s.) is a  quadruple $(X,\mathcal{X},\mu,G)$, where $(X,\mathcal{X},G)$ is a Borel system, $\mu$ is a probability measure on $(X,\mathcal{X})$ so that $\mu$ is $G$-\textbf{invariant}, i.e. for every $g\in G$ it holds that $\mu(gA)=\mu(A)$ for all $A\in\mathcal{X}$. 
An m.p.s. $(X,\mathcal{X},\mu,G)$ is \textbf{ergodic} if for every $A\in\mathcal{X}$ such that $gA=A$ for all $g\in G$, one has $\mu(A)=0$ or $1$. 
An m.p.s. $(X,\mathcal{X},\mu,G)$ is said to be \textbf{free}
 if $\mu(\Xnf)=0$, where\footnote{By Proposition \ref{prop:X_nf is Borel} $\Xnf$ is a Borel set.} $\Xnf:=\{x\in X: \exists g\in G\setminus\{e\},\, \, gx=x\}$.
A map  $\psi:(X,\mathcal{X},\mu,G)\rightarrow (Y,\mathcal{Y},\nu,G)$ is $G$-\textbf{equivariant} on $X'\subset X$ if for all $x\in X'$ and $g\in G$, $\psi(gx)=g\psi(x)$. A Borel map  $\psi:(X,\mathcal{X},\mu,G)\rightarrow (Y,\mathcal{Y},\nu,G)$ is \textbf{measure-preserving} if for every $C\in \mathcal{Y}$, $\mu(\psi^{-1}(C))=\nu(C)$.
A \textbf{measurable factor map} respectively \textbf{measurable isomorphism} is a measure-preserving  Borel map  $\psi:(X,\mathcal{X},\mu,G)\rightarrow (Y,\mathcal{Y},\nu,G)$  such that there exists $X'\subset X$ $G$-invariant Borel subset with $\mu(X')=1$ and  $Y'\subset Y$ $G$-invariant Borel subset with $\nu(Y')=1$ so that $\psi(X')=Y'$ and $\psi$ is $G$-equivariant on $X'$ respectively $G$-equivariant and invertible on $X'$.
The m.p.s.\ $(Y,\mathcal{Y},\nu,G)$ is said to be a  \textbf{measurable factor} of $(X,\mathcal{X},\mu,G)$ respectively  \textbf{isomorphic (as m.p.s.)} to $(X,\mathcal{X},\mu,G)$.

\subsection{Topological dynamical systems}\label{subsec:Topological dynamical systems}

A \textbf{topological dynamical system} (t.d.s.) is a pair $(X,G)$ equipped with a continuous action $G\times X\rightarrow X$, where $X$ is a compact metrizable space and $G$ is a second countable locally compact group. One calls the and refers to $X$ as a $G$-\textbf{system} or a $G$-\textbf{action}; $\R$-systems are sometimes referred to \textbf{topological flows} or as \textbf{topological $\R$-flows}, whereas $\R^k$-systems ($k\geq 2$) are sometimes referred to as \textbf{topological multidimensional flows} or \textbf{topological $\R^k$-flows}. The action of $g\in G$ on $x\in X$, respectively $A\subset X$ is denoted by $gx$, respectively $gA$. 
When $G=\Z$ and the $\Z$ action is generated by the homeomorphism $T:X\rightarrow X$ (for all $x\in X$, $Tx=1 x$), one traditionally writes $(X,T)$. 
A closed $G$-invariant subset of $X$ is called a \textbf{subsystem}. If $G=\R$, this is sometimes referred to as a \textbf{subflow}. A \textbf{fixed point} of $(X,G)$ is a point $x\in X$ for which $gx=x$ for all $g\in G$. We denote by  $\Fix(X,G)$ the set of fixed points of $(X,G)$.  A $G$-system $X$ is said to
be \textbf{free} if $gx=x$ for some $x\in X$ implies $g=e$. The  set of Borel $G$-invariant probability measures on $X$ is denoted by $\mathcal{P}_{G}(X)$.
A morphism between
two dynamical systems $(X,G)$ and $(Y,G)$ is given by a continuous
mapping $\varphi:X\rightarrow Y$ which is $G$-\textbf{equivariant}
($\varphi(gx)=g\varphi(x)$ for all $x\in X$ and $g\in G$). If $\varphi$
is a surjective morphism, $\varphi$ (and sometimes $X)$ is called
an \textbf{extension} and $Y$ is called a \textbf{factor} of $X$,
if $\varphi$ is an injective morphism, it is called an \textbf{(equivariant) embedding}\footnote{This should not be confused with the purely topological context of a continuous injective map $\psi:X\rightarrow Y$ where $X,Y$ are compact metrizable spaces. We will call such a map a \textbf{topological} embedding.}.
The systems $(X,G)$ and $(Y,G)$ are called \textbf{isomorphic} if
$\varphi$ is bijective. 


\subsection{Topological models}
The terminology of \textit{models} in ergodic theory was introduced by Weiss in \cite{w85}, see also \cite[Chapters 2 \& 15]{G03}.
Let $(X,\mathcal{X},\mu,G)$ be an m.p.s. One says that a t.d.s. $(\hat{X},G)$
is a \textbf{topological model} for $(X,\mathcal{X},\mu,G)$  if there exists  $\hat{\mu}\in \mathcal{P}_{G}(\hat{X})$ so that the m.p.s.\
$(X,\mathcal{X},\mu,G)$ is isomorphic (as m.p.s.) to $(\hat{X},\hat{\mathcal{X}},\hat{\mu},G)$, where $\hat{\mathcal{X}}$ is the Borel $\sigma$-algebra of $\hat{X}$. 

\subsection{Upper semicontinuity}\label{subsec:usc}

\begin{defn}
   Let $X$ be a topological space. A map $f:X\to\R\cup\{\infty\}$ is called  \textbf{upper semi-continuous}  if for every $t\in\R$ the set $\{x\in X: f(x)<t\}$ is open in $X$. 
\end{defn}
It is easy to see upper semi-continuous functions as above are always Borel measurable.

\begin{defn}
 Let $X$ be a compact space and denote by $F(X)$ \textbf{the set of all nonempty closed subsets of $X$}. The \textbf{Vietoris topology} on $F(X)$ is the topology generated by the subbasis consisting of sets of the form
\[
\{C\in F(X): C\cap U\neq \emptyset\},\quad 
\{C\in F(X): C\subset U\},
\]
where $U\subset X$ ranges over open subsets.    
\end{defn}
\begin{defn}
    
Let $X$ be topological spaces and $Y$ a compact space. A  map $f:X\to F(Y)$, is called \textbf{upper semi-continuous} if for every $x\in X$ and every open set $U\subset Y$ with $f(x)\subset U$, there exists a neighborhood $V$ of $x$ in $X$ such that $F(x')\subset U$ for all $x'\in V$.
\end{defn}

\begin{thm}\label{thm:usc_implies_Borel}(\cite[Theorem III.9]{CastaingValadier1977})
    Let $X$ be a topological space and $K$ a compact and metric space. An upper semi-continuous function $f:X\to F(K)$ is Borel measurable.  
\end{thm}

\subsection{Metrics}
\begin{defn}
Let $(X,d)$ be a metric space. The metric $d$ is called \textbf{proper} if every closed ball $\overline{B}_r(x)$ for $x\in X$ and $r>0$ is a compact set.
\end{defn}

\begin{defn}
Let $(G,d)$ be a metrizable topological group (where $d$ is a compatible metric on $G$). The metric $d$ is called \textbf{right-invariant}  if for all $h,h',g\in G$ one has $d(hg,h'g)=d(h,h')$.
\end{defn}
\begin{thm}\label{thm:Srtuble}(\cite{struble1974metrics})\label{thm:Struble}
    A locally compact topological group $G$ is metrizable with a proper right-invariant metric iff $G$ is second countable.
\end{thm}

\subsection{Spaces of Lipschitz functions}\label{subsec:Lipschitz functions}
\begin{defn}\label{Lip1}
Let $\tau>0$. A function $f:\mathbb{R}^k\to \R$ is called \textbf{$\tau$-Lipschitz} if 
\begin{equation*}
|f(\ss)-f(\mathbf r)|\leq \tau \|\ss-\mathbf r\|_2,~~\forall~\mathbf{s,r}\in\mathbb{R}^k.
\end{equation*}
Let $\Lip_{1}(\mathbb{R}^k)$
be the space of functions $f:\mathbb{R}^k\to[0,1]$ 
which are $1$-Lipschitz.
We equip it with the following metric, which is compatible with the compact-open topology,
\begin{equation*}
\|f_1-f_2\|_{\Lip_{1}(\mathbb{R}^k)}=\sum\limits_{M=1}^{\infty}2^{-M}\max\limits_{\mathbf s\in[-M,M]^k}|f_1(\ss)-f_2(\ss)|,~~\text{for}~f_1,f_2\in\Lip_{1}(\mathbb{R}^k).
\end{equation*}
\end{defn}

By the Arzel\`{a}--Ascoli theorem $\Lip_{1}(\mathbb{R}^k)$ is compact with respect to this distance. 

More generally let $G$ be a second countable locally compact group and fix a right-invariant compatible metric $d$ on $G$ (see Theorem \ref{thm:Struble}). A function $f:G\to \R$ is called \textbf{$\tau$-Lipschitz} (w.r.t.\ $d$) if 
\begin{equation*}
|f(g)-f(h)|\leq \tau d(g,h),~~\forall~g,h\in G.
\end{equation*}
Let $\Lip_{1}(G)$, the (compact) space of $1$-Lipschitz functions from $G$ to the unit interval (w.r.t. $d$), equipped with the compact-open topology. The group $G$ acts continuously on $\Lip_1(G)$ via the right argument shift (see e.g. \cite[Lemma 1.10]{hjorth1999sharper}), 
$$(gf)(h)=f(hg),~~\forall~g,h\in G,\, f\in \Lip_1(G).$$

\begin{defn}\label{def:Lip(cube)}
    
   Fix $a>0$ and $\tau>0$. Define $\Lip_{\tau}([0,a]^k)$ (respectively $\Lip_{\tau}([0,a]^k,\mathcal{I})$ for some $\mathcal{I}\subset \R$) as the space of maps $\varphi:[0,a]^k\to[0,1]$ (respectively $\varphi:[0,a]^k\to \mathcal{I}$) such that
   \begin{equation*}
       |\varphi(\ss)-\varphi(\mathbf r)|\leq\tau\|\mathbf{s-r}\|_2,~~\forall~\mathbf{s,r}\in[0,a]^k,
   \end{equation*}
   endowed with the distance $\|\varphi-\psi\|_{\infty}=\max\limits_{\mathbf t\in[0,a]^k}|\varphi(\tt)-\psi(\tt)|$.
\end{defn}

\begin{defn}\label{def:orbital_Lipschitz}
Let $(X,\mathcal{B}, G)$ be a Borel system
 where $G$ is a second countable locally compact group.
For a Borel measurable function $\phi:X\to\R$ and a point $x\in X$, define $\phi^{x}:G\to\mathbb{R}$ by
\begin{equation*}\label{notation g^w}
\phi^{x}(g)=\phi(gx).
\end{equation*}
Denote by $\OrbLip_{1}(X,G)=\OrbLip_{1}(X)$ the space of \textbf{orbital Lipschitz functions}, that is Borel measurable functions $\phi:X\to\mathbb{R}$ such that $\phi^{x}\in \Lip_{1}(G)$ for each $x\in X$.\footnote{Orbital Lipschitz functions  are constructed in the proof of Theorem \ref{thm:LCSC_Eberlein}.}
\end{defn}

\begin{lem}\label{measurable}
For each $\phi\in \OrbLip_{1}(X)$, the mapping
\begin{equation}\label{psi}
\psi_{\phi}:x\mapsto \phi^{x}
\end{equation}
determines a (Borel) measurable mapping from $X$ into $\Lip_{1}(G)$. Moreover, $\psi_{\phi}$ is $G$-equivariant in the sense that for each $x\in X$ and $g\in G$:
\begin{equation*}
\psi_{\phi}(gx)=g \psi_{\phi}(x).
\end{equation*}
\end{lem}

\begin{proof}
Let $\phi\in \OrbLip_{1}(X)$. For each $h\in G$, consider the evaluation map
\[
\pi_h:\Lip_1(G)\to \mathbb{R}, \qquad \pi_h(f)=f(h).
\]
Each $\pi_h$ is continuous. Let $H\subset G$ be a countable dense subset (which exists since $G$ is second countable).
Define the map
\[
\Phi:\Lip_1(G)\to \mathbb{R}^H, \qquad \Phi(f) = (f(h))_{h\in H}.
\]
Equipping $\mathbb{R}^H$ with the product topology, $\Phi$ is continuous. Moreover,  $\Phi$ is injective as if $f,g\in \Lip_1(G)$ satisfy $f(h)=g(h)$ for all $h\in H$, then by continuity and density of $H$ in $G$, we obtain $f=g$ on $G$. Thus $\Phi$ is an injective Borel map from a standard Borel space into a  standard Borel space. By Souslin's theorem (\cite[Theorem 2.8(2)]{G03}), $\Phi$ is a Borel isomorphism onto its image. The Borel $\sigma$-algebra on $\mathbb{R}^H$ is generated by the coordinate projections, hence
\[
\mathcal{B}(\Lip_1(G)) = \sigma(\pi_h : h\in H).
\]

Therefore, to prove Borel measurability of $\psi_\phi$, it suffices to check that for each fixed $h\in H$, the map
\[
x \mapsto \pi_h(\phi^x)
\]
is Borel. As
\[
\pi_h(\phi^x) = \phi^x(h) = \phi(hx),
\]
and since the action map $x\mapsto hx$ is Borel and $\phi$ is Borel, it follows that $x\mapsto \phi(hx)$ is Borel.
Finally, for any $x\in X$ and $g,h\in G$, we compute
\[
\psi_\phi(gx)(h)
= \phi^{gx}(h)
= \phi(hgx)
= \phi^x(hg)
= \psi_\phi(x)(hg)
= (g\psi_\phi(x))(h),
\]
which shows that $\psi_\phi$ is $G$-equivariant.
\end{proof}

\subsection{Local sections and complete cross-sections}

In this subsection we introduce various types of \textit{sections} both in topological and measurable categories. 

\begin{defn}\label{def:local section}
    Let $(X,\mathbb{R}^k)$ be a topological $\mathbb{R}^k$-flow and $p\in X$. A \textbf{local section at $p$} is a pair consisting of a  closed set $S\subset X$ containing $p$ and a real number $a>0$  such that the map defined by
    \begin{equation}\label{section}
        [-a,a]^k\times S\to X,~~(\tt,x)\mapsto \mathbf t x
    \end{equation}
    is continuous and injective, and its image contains an open neighborhood of $p$ in $X$. 
\end{defn}

\begin{thm}\label{thm:local section}
    Let $(X,\mathbb{R}^k)$ be a weakly locally free topological multidimensional flow as defined in Definition \ref{def:LFOF}. Then for any $p\in X\setminus\Fix(X,\mathbb{R}^k)$, there exists a local section at $p$.
\end{thm}
The proof is given in Appendix A.

\begin{defn}\label{def:Borel cross-section}(\cite{Kechris1992})
  Let $(X,\mathcal{B}, G)$ be a Borel system where $G$ is a second countable locally compact group. A Borel subset $Z\subset X$ is called a  $G$-\textbf{complete cross-section} if it intersects every $G$-orbit in a countable non-empty set.  If in addition  there exists an open neighborhood $e\in U\subset G$ so that for all $z\in Z$,
  $Uz\cap Z=\{z\}$ then $Z$ is called a \textbf{complete \mbox{($U$-)lacunary} $G$-cross-section}.
\end{defn}

Notice the following theorem does not require a freeness assumption.
\begin{thm}\label{thm:lacunary cross-section} (Kechris \cite[Corollary 1.2]{Kechris1992})
    Let $(X,\mathcal{B}, G)$ be a Borel system where $G$ is a second countable locally compact group. Then it admits  a complete lacunary $G$-cross-section. 
\end{thm}
\begin{rem}
 In the context of  non-singular (in particular measure-preserving) actions by second countable locally compact group, an older result by Feldman, Hahn and Moore \cite[Theorem 2.8]{FeldmanHahnMoore1979} had shown the existence of an  \textit{almost-surely complete} lacunary cross-section. For  free $\Rk$-Borel systems Slutsky  (\cite[Theorem 4.3]{Slu23}) generalized Katok's representation theorem (\cite{Kat77}) which allows to find a  complete lacunary $\Rk$-cross-section with additional structure.    
\end{rem}

\section{A generalization of Eberlein's theorem to second countable locally compact Borel systems}

\subsection{Auxiliary lemma}  

We present a lemma that is well known to specialists. Several\footnote{For example, the lemma may be established using the Kuratowski–Ryll-Nardzewski measurable selection theorem.} proofs are available in the literature; for completeness, we include one for the sake of completeness.

\begin{lemma}\label{lem:rho_is_Borel}
Let $G$ be a second countable locally compact group with a proper
right-invariant compatible metric $d$, and let $(X,\mathcal{B}, G)$ be a Borel system. Then the map
\[
\rho:X\times X\to [0,\infty],\qquad 
\rho(x,y)=\inf\{d(g,e): gx=y\},
\]
where the infimum of the empty set is taken to be $\infty$, is a  Borel measurable extended metric. 

\end{lemma}

\begin{proof}

By a theorem of Becker and Kechris \cite[Theorem 5.2.1]{BK96} we may assume w.l.o.g. that $X$ is a Polish space and the action of $G$ on $X$ is continuous.  Fix $r>0$ and define:
$$
R_r = \{(x,y,g)\in X\times X\times \overline{B}_r(e)\, |\, gx=y\}.
$$
Let $\pi_{X\times X}:X\times X\times G\rightarrow X\times X$ be the projection on the first two coordinates and $\pi_G:X\times X\times G\rightarrow G$ be the projection on the  third coordinate. Let $F(\overline{B}_r(e))$ denote the set of non-empty closed sets of  $\overline{B}_r(e)$. Using the fact that $\overline{B}_r(e)$ is compact, it is easy to see that $\pi_{X\times X}(R_r)$ is closed and the map
$$
s_r: \pi_{X\times X}(R_r)\rightarrow F(\overline{B}_r(e)),\, (x,y)\mapsto  \pi_G(R_r\cap \{(x,y)\}\times G) 
$$
is upper semi-continuous and thus by Theorem \ref{thm:usc_implies_Borel}, Borel measurable. Indeed, fix $(x,y)\in\pi_{X\times X}(R_r)$, and let $U\subset\overline B_r(e)$ be open with $s_r(x,y)\subset U$. Set $C:=\overline B_r(e)\setminus U$, which is closed (hence compact). If $s_r$ were not upper semi-continuous at $(x,y)$, there would exist $(x_n,y_n)\to(x,y)$ and $g_n\in s_r(x_n,y_n)\cap C$ for all $n$. By compactness of $C$ pass to a subsequence with $g_{n_k}\to g\in C$. As $R_r$ is closed, conclude $(x,y,g)\in R_r$, so $g\in s_r(x,y)\cap C$, contradicting $s_r(x,y)\subset U$. Hence $s_r$ is upper semi-continuous.
Define the map 
$$
D_r: F(\overline{B}_r(e))\rightarrow [0,\infty],\, C\mapsto \inf_{g\in C} d(g,e)
$$
Note that for every $t\in \R$,
$$
D_r^{-1}((-\infty,t))=\{C\in F(\overline{B}_r(e)): C\cap B_t(e)\neq\emptyset\}.$$
where the RHS is an open set in the Vietoris topology. Conclude that  $D_r$ upper semi-continuous and thus Borel measurable. 

For each $r>0$ let $F_{\star}(\overline{B}_r(e)):=F(\overline{B}_r(e))\cup \{\star\}$ where $\star$ is added as an isolated point. Extend $s_r$ to $\tilde{s}_r:X\times X\rightarrow F_{\star}(\overline{B}_r(e))$ by assigning $\star$ to the values outside $\pi_{X\times X}(R_r)$. Clearly $\tilde{s}_r$ is  Borel measurable. Extend $D_r$ to $\tilde{D}_r:F_{\star}(\overline{B}_r(e))\to [0,\infty]$ by $\tilde{D}_r(\star)=\infty$. Clearly $\tilde{D}_r$ is  Borel measurable. 

Finally we note it holds for all $x,y\in X$,
$$
\rho(x,y)=\inf_{n\in\N} \tilde{D}_n(\tilde{s}_n(x,y)).
$$
Thus $\rho$ being the infimum over a countable collection of Borel functions is Borel.

We now verify the axioms of an extended metric.

Clearly $\rho(x,x)=0$ for all $x\in X$. Let $x,y\in X$ so that $\rho(x,y)=0$. By definition we may find a sequence $g_n\in G$ so that $g_n\rightarrow e$ and $g_n x=y$. By the continuity of the action $x=y$.

To prove symmetry fix $x,y\in X$. Note that for all $g\in G$ $gx=y \Leftrightarrow g^{-1}y=x$ and by the symmetry and right-invariance of $d$ it holds $d(g,e)=d(e,g^{-1})=d(g^{-1},e)$. Thus it follows $\rho(x,y)=\rho(y,x)$.

To prove the triangle inequality
fix $x,y,z\in X$. W.l.o.g. we may assume $\rho(x,y)<\infty$ and $\rho(y,z)<\infty$. Thus we may find $g,h\in G$ with
$gx=y$ and $hy=z$ so that $
d(g,e)\le \rho(x,y)+\epsilon$ and $d(h,e)\le \rho(y,z)+\epsilon$. Note $hgx=hy=z$ and therefore
$$
\rho(x,z)\le d(hg,e) \le d(hg,g)+d(g,e) = d(h,e)+d(g,e) \le \rho(y,z)+\rho(x,y)+2\epsilon.
$$
Taking $\epsilon\rightarrow 0$ yields the desired inequality.

\end{proof}

\subsection{Proof of the theorem}

\begin{proof}[Proof of Theorem \ref{thm:LCSC_Eberlein}]
By Theorem \ref{thm:Struble}  we may fix a proper right-invariant compatible metric  $d$ on $G$. 
    Let $\rho: X\times X\to\R_+\cup\{\infty\}$ be given by
    \begin{equation*}
        \rho(x,y)=\inf\{d(g,e): g\in G,gx=y\},
    \end{equation*}
    where the infimum of the empty set is taken to be $\infty$. By Lemma \ref{lem:rho_is_Borel} $\rho$ is a  Borel extended metric.

    Note that by right-invariance of $d$ and the fact that the element $hg^{-1}$ maps $gx$ to $hx$, it holds: 
    \begin{equation}\label{eq:rho_to_d}
      \rho(gx,hx)\leq d(hg^{-1},e)= d(g,h) ~~~~\forall x\in X,~\forall g,h\in G.
    \end{equation}
From the reverse triangle inequality for $\rho$ it follows (for the case of $x$ and $y$ not belonging to the same orbit, we employ the convention $\infty-\infty=0$):
    \begin{equation}\label{Lipschitz diatance}
        |\rho(gx,y)-\rho(hx,y)|\leq \rho(gx,hx) \stackrel{\eqref{eq:rho_to_d}}{\leq} d(g,h) ~~~~\forall x,y\in X,~\forall g,h\in G.
    \end{equation}
    By Theorem \ref{thm:lacunary cross-section}, there is a complete $B_{4r}(e)$-lacunary Borel cross-section $S\subset X$ for some $r>0$. By rescaling $d$, one may assume w.l.o.g.\footnote{A detailed analysis of the proof given in \cite{Kechris1992} shows one can guarantee the existence of a $B_{4}(e)$-lacunary Borel cross-section  without any rescaling, however the details are beyond the scope of this paper.} $r=1$.  Thus it holds  $G\cdot S=X$ and $(B_{4}(e) s)\cap S=\{s\}$ for all $s\in S$. In particular given two distinct elements $s_1,s_2\in S$, either they are not in the same orbit and then $\rho(s_1,s_2)=\infty$ or they are in the same orbit, that is $s_1\neq s_2$ and $gs_1=s_2$ for some $g\in G$, and then $g\notin B_{4}(e)$ implying 
    \begin{equation}\label{eq:distinct_s}
       \rho(s_1,s_2)\geq 4.
    \end{equation}
     We claim that if $x\in B_1(e)\cdot S$, there exists a unique $s\in S$ such that $x\in B_1(e) s$. We denote this $s$ by $\pi(x)$. Indeed, if there exist distinct $s,s'\in S$ satisfying $x\in (B_1(e) s)\cap(B_1(e) s')$, i.e. $x=gs=g's'$ for some $g,g'\in B_1(e)$, then by the triangle inequality for $\rho$, $\rho(s,s')\leq \rho(x,s)+\rho(x,s')<2$, which contradicts $\rho(s,s')\geq 4$.\\
      \noindent
  \textbf{Claim.} The set $B_1(e)\cdot S$ is a Borel set and the map $\pi:B_1(e)\cdot S\rightarrow S$ is a Borel map.\\
   \noindent
 \textbf{Proof of the Claim.} Consider the Borel set 
  $$\Gamma:=\{(x,s)\in X\times S|\,\rho(x,s)<1\}\subset X\times S.$$
  Note
  \begin{equation}\label{eq:Gamma_char}
  \Gamma=\{(x,\pi(x))|\,x\in B_1(e)\cdot S\}.       \end{equation}
 Indeed if 
 $x\in B_1(e)\cdot S$, then  there exists $s\in S$ and $g\in B_1(e)$ such that $x = gs$ (and necessarily $s=\pi(x)$). Conclude $\rho(x,s) = \rho(gs,s) \stackrel{\eqref{eq:rho_to_d}}{\leq}  d(g,e) < 1$. Conversely, if $\rho(x,s)<1$ for some $(x,s)\in X\times S$, then by definition there exists $g\in B_1(e)$ so that $gx=s$. Conclude $x=g^{-1}s\in B_1(e)\cdot S$ and $s=\pi(x)$. By the Lusin–Novikov theorem (\cite[Theorem 18.10]{Kec95}), as for all $x\in X$, $|\Gamma_x|\leq 1$ where $\Gamma_x:=\{s\in S|\, (x,s)\in \Gamma\}$, $\Gamma$ is a \textit{Borel graph}, which given \eqref{eq:Gamma_char}, is equivalent to the statement of the claim.  \hfill  $\blacksquare$

    By \cite[Theorem 15.6]{Kec95} there is a Borel injection $\alpha: S\to[1,2]$. Define the Borel function $\phi:X\to[0,1]$ by\footnote{The fact that $\phi$ takes values in $[0,1]$ follows from \eqref{eq:Gamma_char}.}
     \begin{equation*}
         \phi(x)=
         \begin{cases}
         \frac{1-\rho(x,\pi(x))}{\alpha(\pi(x))} & \textrm{if }\, x\in B_1(e)\cdot S,\\
         0   &    \textrm{if }\, x\notin B_1(e)\cdot S.
       \end{cases}
     \end{equation*}
      
      \noindent
      \textbf{Claim.} The function $\phi$ is an orbital Lipschitz function in the sense of Definition  \ref{def:orbital_Lipschitz}.\\
      \textbf{Proof of the Claim.} By definition, one has to show for every $x\in X$, $\phi^x\in\Lip_1(G)$ , where $\phi^x:G\to[0,1]$ is defined by $g\mapsto \phi(gx)$.           
      Let $x\in X$ and $g,h\in G$. We treat several cases:
      \begin{enumerate}
          \item 
      For some $s\in S$, $gx,hx\in B_1(e) s$. Thus $\pi(gx)=\pi(hx)=s$. By \eqref{Lipschitz diatance}, it holds
      \begin{equation*}
          \begin{split}
              |\phi^x(g)-\phi^x(h)|&=|\phi(gx)-\phi(hx)|=\left|\frac{1-\rho(gx,\pi(gx))}{\alpha(\pi(gx))}-\frac{1-\rho(hx,\pi(hx))}{\alpha(\pi(hx))}\right|\\
              &=\frac{|\rho(hx,s)-\rho(gx,s)|}{\alpha(s)}\leq d(g,h).
          \end{split}
      \end{equation*}
     
      \item
      For  distinct $s_1,s_2\in S$, $gx\in B_1(e) s_1,hx\in B_1(e) s_2$. Thus $\pi(gx)=s_1$ and $\pi(hx)=s_2$.  Note  $\rho(gx,s_1)<1$ and $\rho(hx,s_2)<1$. Assume for a contradiction that $d(g,h)<2$. By the triangle inequality
      $$\rho(s_1,s_2)\leq \rho(s_1,gx)+\rho(gx,hx)+\rho(hx,s_2)\stackrel{\eqref{eq:rho_to_d}}{<}1+1+2<4.$$
      This contradicts the fact that $S$ is $B_{4}(e)$-lacunary. Thus $d(g,h)\geq 2$.

      By \eqref{eq:rho_to_d} and as $\alpha\geq 1$,
                 \begin{equation*}
\begin{split}
|\phi^x(g)-\phi^x(h)| 
&= \left|\frac{1-\rho(gx,s_1)}{\alpha(s_1)}-\frac{1-\rho(hx,s_2)}{\alpha(s_2)}\right| \\
&\leq \left|\tfrac{1-\rho(gx,s_1)}{\alpha(s_1)}\right|
     + \left|\tfrac{1-\rho(hx,s_2)}{\alpha(s_2)}\right|
     \leq \tfrac{2}{1}\leq d(g,h).
\end{split}
\end{equation*}

      \item
           For some $s\in S$, $gx\in B_1(e) s$ and $hx\notin B_1(e)\cdot S$. Thus $\pi(gx)=s$. As $\rho(gx,s)<1$ and $\rho(hx,s)\geq 1$, it holds by \eqref{Lipschitz diatance}
      \begin{equation*}
          \begin{split}
              |\phi^x(g)-\phi^x(h)|=\frac{1-\rho(gx,s)}{\alpha(s)}-0\leq\frac{\rho(hx,s)-\rho(gx,s)}{\alpha(s)}\leq d(g,h).
          \end{split}
      \end{equation*}
      \end{enumerate}
      We conclude that $|\phi^x(g)-\phi^x(h)|\leq d(g,h)$ for all $g,h\in G$ as desired. \hfill $\blacksquare$

      Define $\Phi:X\to\Lip_1(G)$ by $\Phi(x)=\phi^x$. By Lemma \ref{measurable}, $\Phi$ is a $G$-equivariant Borel map. \\
      \noindent
      \textbf{Claim.} The function $\Phi$ is injective.\\
      \textbf{Proof of the Claim.}  
     Let $x,y\in X$ and assume $\phi^x=\Phi(x)=\Phi(y)=\phi^y$.  As $S$ is a complete Borel cross-section and $\Phi$ is $G$-equivariant we may assume w.l.o.g.\ $x\in S$. If $y\in S$ then $\frac{1}{\alpha(x)}=\phi^x(e)=\phi^y(e)=\frac{1}{\alpha(y)}$. Thus in this case $x=y$ follows from  the injectivity of $\alpha$. Now assume for a contradiction $y\notin S$. As  $\phi^x(e)>0$, one must have $y\in B_1(e)\cdot S$ (otherwise $\phi^y(e)=0$). Let $s=\pi(y)\neq y$. As $\rho$ is an extended metric (see Lemma \ref{lem:rho_is_Borel}), it is clear $\phi(s)>\phi(y)$. Thus there exists $h_0\in B_1(e)$ so that $\phi^y(h_0)>\phi^y(e)$ (e.g., $h_0\in B_1(e)$ for which   $s=h_0y$). In contrast by the same reasoning for all  $h\in B_1(e)$, $\phi^x(h)\leq\phi^x(e)$. This contradicts $\phi^x=\phi^y$ and the proof is completed.      
      \hfill $\blacksquare$

   \noindent  
   We have thus constructed an injective $G$-equivariant Borel map $\Phi:X\to\Lip_1(G)$ as desired.
\end{proof}

\section{Gutman--Jin--Tsukamoto embedding theorem for multidimensional flows}\label{BK}

\subsection{Proof of the theorem}

\begin{proof}[Proof of Theorem \ref{RkGJT} assuming Theorems \ref{Rkdense1} \& \ref{Rkdense2} from Subsection \ref{subsec:Density theorems} below]

The proof uses a Baire category theorem argument. By assumption there exists a (topological) embedding $h_0:\Fix(X,\Rk)\hookrightarrow\Fix(\Lip_{1}(\Rk))\cong[0,1]$. Define 
\begin{equation*}
    C_{\Rk,h_0}(X;\Lip_{1}(\Rk))=\{f\in C(X;\Lip_{1}(\Rk)): f~\text{is}~\Rk-\equi,~f|_{\Fix(X,\Rk)}=h_0\},
\end{equation*}   
endowed with the compact-open topology. By Lemma \ref{lemma:non-empty} in the sequel this space is non-empty. Define

\begin{equation*}
       \Omega:=(X\setminus\Fix(X,\Rk))\times(X\setminus\Fix(X,\Rk))\setminus\{(x,x)| x\in X\}. 
    \end{equation*}
     For a closed set $A\subset X\setminus\Fix(X,\Rk)$, define
\begin{equation}\label{G(A)}
    G(A)=\{f\in C_{\Rk,h_0}(X;\Lip_{1}(\Rk)): f(A)\cap\Fix(\Lip_{1}(\Rk))=\emptyset\}.
\end{equation}
For closed sets $B$ and $C$ such that $B\times C \subset \Omega$, define
   \begin{equation}\label{G(B,C)}
    G(B,C)=\{f\in C_{\Rk,h_0}(X;\Lip_{1}(\Rk)): f(B)\cap f(C)=\emptyset\}.
\end{equation}
     By Theorem \ref{Rkdense1}  for any $p\in X\setminus\Fix(X,\Rk)$, there exists a closed neighborhood $A$ of $p$ in $X$ such that $G(A)$
is open and dense in $C_{\Rk,h_0}(X;\Lip_{1}(\Rk))$. By  Theorem  \ref{Rkdense2}  for any two distinct points $p,q\in X\setminus\Fix(X,\Rk)$, there exist closed neighborhoods $B$ and $C$ of $p$ and $q$ in $X$ respectively such that $G(B,C)$
is open and dense in $C_{\Rk,h_0}(X;\Lip_{1}(\Rk))$.
     
 As $X$ is a compact metric space, it is well known that $X$ and $X\times X$ are \textit{hereditarily Lindelöf spaces}\footnote{A \textbf{Lindelöf space} is a topological space in which every open cover has a countable subcover. A \textbf{hereditarily Lindelöf space} is a topological space such that every subspace of it is Lindelöf.}. Thus there exist countable families of closed subsets $\{A_n\}_{n=1}^{\infty},\{B_n\}_{n=1}^{\infty}$ and $\{C_n\}_{n=1}^{\infty}$ of $X$ such that     $X\setminus\Fix(X,\Rk)=\bigcup_{n=1}^{\infty}A_n$, $\Omega=\bigcup_{n=1}^{\infty}B_n\times C_n$ and   \begin{itemize}
         
    \item $G(A_n)$ is open and dense in $C_{\Rk,h_0}(X;\Lip_{1}(\Rk))$ for each $n\geq1$.
    \item $G(B_n,C_n)$ is open and dense in $C_{\Rk,h_0}(X;\Lip_{1}(\Rk))$ for each $n\geq1$.
\end{itemize}
    
By the Baire category theorem, the set $\bigcap\limits_{n=1}^{\infty}G(A_n)\cap\bigcap\limits_{n=1}^{\infty}G(B_n,C_n)$ is dense in $C_{\Rk,h_0}(X;\Lip_{1}(\Rk))$ and in particular non-empty. Fix $f\in \bigcap\limits_{n=1}^{\infty}G(A_n)\cap\bigcap\limits_{n=1}^{\infty}G(B_n,C_n)$. Let  $x,y\in X$ so that $x\neq y$. We verify $f(x)\neq f(y)$, which implies $f$ is an equivariant embedding of  $(X,\Rk)$ into $\Lip_{1}(\Rk)$, by considering several cases. If $x,y\in \Fix(X,\Rk)$, then $f(x)=h_0(x)\neq h_0(y)=f(y)$. If $x\notin \Fix(X,\Rk)$ and $y\in \Fix(X,\Rk)$, then there exists $n\in \N$ so that $x\in A_n$ and therefore $ f(A_n)\ni f(x)\neq h_0(y)= f(y)$ as $f(A_n)\cap\Fix(\Lip_{1}(\Rk))=\emptyset$.   Finally if   $x,y\notin \Fix(X,\Rk)$, then there exists $n\in \N$ so that $(x,y)\in B_n\times C_n$ and therefore $ f(B_n)\ni f(x)\neq f(y)\in f(C_n)$ as $f(B_n)\cap f(C_n)=\emptyset$.
 Q.E.D.

\end{proof}

\begin{defn}
For $f\in C_{\Rk}(X;\Lip_{1}(\Rk))$ we define $\Lip(f)$ as the supremum of
    \begin{equation*}
        \frac{|f(x)(\tt)-f(x)(\ss)|}{\|\mathbf {t-s}\|_2}
    \end{equation*}
    over all $x\in X$ and $\mathbf {s,t}\in\Rk$ with $\mathbf s\neq \mathbf t$.
    \end{defn}
The following lemma is used in the proof of Theorems \ref{Rkdense1} and \ref{Rkdense2}. 

\begin{lem}\label{lemma:non-empty}(cf. \cite[Lemma 3.1]{GJT19})\label{GJTLemma3.1}
        There exists $f_0\in C_{\Rk,h_0}(X;\Lip_{1}(\Rk))$ such that $\Lip(f_0)\leq\frac{1}{2}$. In particular $C_{\Rk,h_0}(X;\Lip_{1}(\Rk))\neq \emptyset$.
    \end{lem}
    \begin{proof}
        
         Consider the map
         \begin{equation*}
            h_0:\Fix(X,\Rk)\to[0,1].
         \end{equation*}
        By the Tietze extension theorem, see e.g. \cite{Tie15}, we may extend this function to a continuous map $h_1: X\to[0,1]$. 
Recall $\int_{\mathbb{R}}e^{-\alpha t^2}dt=\sqrt{\frac{\pi}{\alpha}}$ and $\int_{\mathbb{R}}|t|e^{-\alpha t^2}dt=\frac{1}{\alpha}$ for $\alpha>0$.        
                       Fix $b\geq\frac{4k}{\sqrt{\pi}}$ and let $\varphi:\mathbb{R}^k\to[0,1]$ be a smooth function defined by $\varphi(\tt)=\frac{1}{(\sqrt{\pi} b)^k}e^{-\frac{\sum_{i=1}^{k}t_i^2}{b^2}}$ for $\mathbf t\in\mathbb{R}^k$. Then
        \begin{equation*}
            \int_{\mathbb{R}^k}\varphi(\tt)d\mathbf t=\frac{1}{(\sqrt{\pi} b)^k}\int_{\mathbb{R}} e^{-\frac{t_1^2}{b^2}}dt_1\cdots\int_{\mathbb{R}}e^{-\frac{t_k^2}{b^2}}dt_k=\frac{1}{(\sqrt{\pi} b)^k}(\sqrt{\pi} b)^k=1
        \end{equation*}
        and
        \begin{equation*}
        \begin{split}
            \int_{\mathbb{R}^k}\|\nabla\varphi(\tt)\|_1 d\mathbf t&=\frac{1}{(\sqrt{\pi} b)^k}\frac{2}{ b^2}\int_{\Rk}(|t_1|+\cdots+|t_k|)e^{-\frac{t_1^2+\cdots+t_k^2}{b^2}}dt_1\cdots dt_k\\
            &=\frac{2}{\pi^{k/2} b^{k+2}}\int_{\mathbb{R}}|t_1|\cdot e^{-\frac{t_1^2}{b^2}}dt_1\cdot\int_{\mathbb{R}^{k-1}}e^{-\frac{\sum_{i=2}^k t_i^2}{b^2}}dt_2\cdots dt_k+\cdots\\ &+\frac{2}{\pi^{k/2} b^{k+2}}\int_{\mathbb{R}}|t_k|\cdot e^{-\frac{t_k^2}{b^2}}dt_2\cdot\int_{\mathbb{R}^{k-1}}e^{-\frac{\sum_{i=1}^{k-1} t_i^2}{b^2}}dt_1\cdots dt_{k-1}\\
            &=k\frac{2b^2 (\sqrt{\pi} b)^{k-1}}{\pi^{k/2} b^{k+2}}=\frac{2k}{b\sqrt{\pi}}
                        \leq\frac{1}{2}.
        \end{split}
        \end{equation*}
          For $x\in X$ and $\tt\in \Rk$ we define  
        \begin{equation*}
            f_0(x)(\tt)=\int_{\Rk}\varphi(\mathbf {t-s})h_1(\ss x)d\mathbf s.
        \end{equation*}
 Note $0\leq f_0(x)(\tt)\leq \int_{\Rk}\varphi(\mathbf {t-s})d\mathbf s=1$. Direct calculation shows that 
        \begin{equation*}
            \nabla f_0(x)(\tt)=\left(\frac{\partial f_0(x)}{\partial t_1}(\tt),\ldots,\frac{\partial f_0(x)}{\partial t_k}(\tt)\right)=\left(\int_{\Rk}\frac{\partial\varphi}{\partial t_1}(\mathbf {t-s})h_1(\ss x)d\mathbf s,\ldots, \int_{\Rk}\frac{\partial\varphi}{\partial t_k}(\mathbf {t-s})h_1(\ss x)d\mathbf s\right).
        \end{equation*}
        Since $0\leq h_1(x)\leq1$ for all $x\in X$, $\|\nabla f_0(x)(\tt)\|_1\leq\int_{\Rk}\|\nabla\varphi(\ss)\|_1 d\mathbf s\leq\frac{1}{2}$. By Lagrange's mean value theorem, for $\mathbf {s,t}\in\Rk$ with $\mathbf s\neq \mathbf t$, there exists $\theta\in(0,1)$ such that
        \begin{equation*}
            f_0(x)(\tt)-f_0(x)(\ss)=\langle \tt-\ss, \nabla f_0(x) ((1-\theta)\ss+\theta \tt)\rangle. 
        \end{equation*}
        Then using the Cauchy--Schwarz inequality  and the fact that $\|\cdot\|_2\leq \|\cdot\|_1 $ in Euclidean space, it holds        
       
        $$|f_0(x)(\tt)-f_0(x)(\ss)|\leq\|\tt-\mathbf s\|_2\cdot\|\nabla f_0(x)((1-\theta)\mathbf s+\theta \mathbf t)\|_2\leq\|\tt-\mathbf s\|_2\cdot\|\nabla f_0(x)((1-\theta)\mathbf s+\theta \mathbf t)\|_1\leq\frac{1}{2}\|\mathbf {t-s}\|_2.$$
                Thus $\Lip(f_0)\leq\frac{1}{2}$ and $f_0\in C(X;\Lip_{1}(\Rk))$. Note that $f_0:X\to\Lip_{1}(\Rk)$ is $\Rk$-equivariant, as $f_0(\rr x)(\tt)=\int_{\Rk}\varphi(\mathbf {t-s})h_1((\ss+\rr) x)d\mathbf s=f_0(x)(\mathbf {t+r})$ holds for all $x\in X$ and $\mathbf {t,r}\in\Rk$.
        Furthermore, if $x\in\Fix(X,\Rk)$, $\ss x=x$ for all $\mathbf s\in\Rk$ and therefore $h_1(\ss x)=h_1(x)=h_0(x)$. Hence $f_0|_{\Fix(X,\Rk)}=h_0$ since $\int_{\Rk}\varphi(\ss)d\mathbf s=1$.
               Thus we conclude that $f_0\in C_{\Rk,h_0}(X;\Lip_{1}(\Rk))$.
    \end{proof}

\subsection{Density theorems}\label{subsec:Density theorems}

\begin{thm}\label{Rkdense1}(cf. \cite[Lemma 3.3]{GJT19})
Let $(X,\Rk)$ be a  weakly locally free topological multidimensional flow and  $h_0:\Fix(X,\Rk)\hookrightarrow[0,1]$ be a topological  embedding.  
Then for any $p\in X\setminus\Fix(X,\Rk)$, there exists a closed neighborhood $A$ of $p$ in $X$ such that $G(A)$\footnote{See Equation \eqref{G(A)}.}
is open and dense in $C_{\Rk,h_0}(X;\Lip_{1}(\Rk))$.
\end{thm}
\begin{proof}
    Fix $p\in X\setminus\Fix(X,\Rk)$. By Theorem \ref{thm:local section}, there is a local section $(a,S)$ around $p$. We choose a closed neighborhood $A_0$ of $p$ in $S$ satisfying $A_0\subset\Int([-a,a]^k\cdot S)$. We define a closed neighborhood $A$ of $p$ in $X$ by
    \begin{equation*}
        A=\bigcup\limits_{\tt\in[-a,a]^k} \tt A_0.
    \end{equation*}
       
\noindent   Notice that $A$ is a neighborhood of $p$ as $(a,S)$ is a local section around $p$.  It is obvious that the set $G(A)$ defined by Equation \eqref{G(A)} is open in $C_{\Rk,h_0}(X;\Lip_{1}(\Rk))$. We now prove that $G(A)$ is also dense in $C_{\Rk,h_0}(X;\Lip_{1}(\Rk))$. 
     Fix $f\in C_{\Rk,h_0}(X;\Lip_{1}(\Rk))$ and $0<\delta<1$. Let $f_0$ be given by Lemma \ref{GJTLemma3.1}. We define $f_1\in C_{\Rk,h_0}(X;\Lip_{1}(\Rk))$ by
    \begin{equation}\label{delta perturb}
        f_1(x)(\tt)=(1-\delta)f(x)(\tt)+\delta f_0(x)(\tt).
    \end{equation}
    It follows $\tau:=\Lip(f_1)\leq1-\frac{\delta}{2}<1$. We apply Lemma \ref{main lemma} from Subsection \ref{subsec:main_lemma} below  to the map
    \begin{equation*}
        X\ni x\mapsto f_1(x)|_{[0,a]^k}\in\Lip_{\tau}([0,a]^k)
    \end{equation*}
    and find $\AA \subset (0,a)^k$ a finite non-empty subset and $g\in C(X;\Lip_{1}([0,a]^k))$ such that
    \begin{enumerate}[$(1)$]
    \item  $|g(x)(\tt)-f_1(x)(\tt)|<2\delta$ for all $x\in X$ and $\mathbf t\in[0,a]^k$.
    \item  For all $x\in X$, $g(x)(\tt)=f_1(x)(\tt)$ if $\mathbf t\in\Edge([0,a]^k)$.
   \item $g(x)_{|\AA}$ is a non-constant function for all $x\in X$.
    \end{enumerate}

\noindent     We now define a perturbation $g_1:X\to\Lip_{1}(\Rk)$ of $f_1\in C_{\Rk,h_0}(X;\Lip_{1}(\Rk))$ (which itself is a perturbation of $f$). For each $x\in X$, define $H(x):=\{\mathbf t\in\Rk|\,\, \tt x\in S\}$. The set $H(x)$ is best thought of as a set of \textit{markers} on the orbit of $x$. As $(a,S)$ is a local section around $p$, any two distinct $\mathbf {t,s}\in H(x)$ satisfy $\tt-\ss\notin [-a,a]^k$. Indeed if  $\tt-\ss\in [-a,a]^k$ then the map $(\rr,x)\mapsto \rr  x$ is not injective on $[-a,a]^k\times S$ as the pairs $(\tt-\ss, \ss x)$ and  $(\00, \tt x)$ have the same image. Thus the markers on each orbit are separated.
Recall  $A_0$ is a closed neighborhood of $p$ in $S$ satisfying $A_0\subset\Int([-a,a]^k\cdot S)$. Our goal is to change $f_1(x)$ for $x\in A$ at locations  $\ss\in H(x)$ with $\ss x \in A_0$ (by the definition of $\AA$ such $\ss$ exists). Specifically for such $\ss$ and $\tt\in [0,a]^k$ we would like to replace $f_1(\ss x)(\tt)$ by $g(\ss x)(\tt)$. Note that such a replacement results with a $1$-Lipschitz (in particular continuous function). Indeed  for $\mathbf t\in \Edge([0,a]^k)$, it holds $g(\ss x)(\tt)=f_1(\ss x)(\tt)=f_1(x)(\tt+\ss)$ by the second property of $g$. Observe however that the map $x\mapsto H(x)$ is not continuous in any meaningful way. Thus in order to make the replacement process we just described continuous in the variable $x$, we introduce \textit{weights}. 
Choose a continuous function $q:S\to[0,1]$ such that $q=1$ on $A_0$ and $\supp(q)\subset\Int([-a,a]^k\cdot S)$.
Let $x\in X$.

 For $t\in[0,a]^k$ set
$$u(x)(\tt)=g(x)(\tt)-f_1(x)(\tt).$$
    \begin{itemize}
        \item For each $\mathbf s\in H(x)$ and $\mathbf t\in[s_1,s_1+a]\times \cdots \times [s_k,s_k+a]$, set
        \begin{equation*}
            g_1(x)(\tt)=f_1(x)(\tt)+q(\ss x)\cdot u(\ss x)(\mathbf {t-s}).
        \end{equation*}
        \item For $\mathbf t\in\Rk\setminus\bigcup\limits_{\mathbf s\in H(x)}[s_1,s_1+a]\times \cdots \times [s_k,s_k+a]$, set $g_1(x)(\tt)=f_1(x)(\tt)$.
    \end{itemize}
    Then $g_1(x)$ is well-defined since any two distinct $\mathbf {t,s}\in H(x)$ satisfy $\tt-\ss\notin [-a,a]^k$. Note that as desired if  $\ss x\in A_0$ then for $\tt\in [0,a]^k$ 
    $$
g_1(\ss x)(\tt)=f_1(\ss x)(\tt)+1\cdot u(\ss x)(\tt)=g(\ss x)(\tt).
    $$
    \begin{claim}\label{equi and perturb}
It holds that  $g_1\in C_{\Rk,h_0}(X;\Lip_{1}(\Rk))$ and $\max\limits_{x\in X}\max\limits_{\mathbf t\in\Rk}|g_1(x)(\tt)-f(x)(\tt)|\leq 4\delta$. 
    \end{claim}
    \begin{proof}
    It is not hard to see $g_1$ is continuous in the variable $x$. Indeed the key observation is that if $x_i\rightarrow x$ and $\ss \in H(x)$ such that $\ss x \in \supp(q)\subset\Int([-a,a]^k\cdot S)$, then one may find $\ss_i \in \Rk$ such that $\ss_i\rightarrow \ss$ and $\ss_i \in H(x_i)$.  Note that $H(x)=\emptyset$ if $x\in\Fix(X,\Rk)$. Thus, $g_1|_{\Fix(X,\Rk)}=f_1|_{\Fix(X,\Rk)}=h_0$. 
    \noindent
       Let $x\in X$ and $\rr\in \Rk$, if $\mathbf s\in H(\rr x)$, then for $\mathbf t\in[s_1,s_1+a]\times \cdots \times [s_k,s_k+a]$ it holds
    \begin{equation}\label{LHS}
        g_1(\rr x)(\tt)=f_1(\rr x)(\tt)+q((\ss+\rr)x)\cdot u((\ss+\rr)x)(\mathbf {t-s}).
    \end{equation}
    Note that $\mathbf s\in H(\rr x)$ implies that $(\ss+\rr)x\in S$ and therefore $\ss+\rr\in H(x)$ and $\tt+\rr\in [s_1+r_1,s_1+r_1+a]\times \cdots \times [s_k+r_k,s_k+r_k+a]$. Thus
    \begin{equation*}
        g_1(x)(\mathbf {t+r})=f_1(x)(\mathbf {t+r})+q((\ss+\rr)x)\cdot u((\ss+\rr)x)((\mathbf {t+r})-(\mathbf {s+r}))~\text{for}~\mathbf t\in[s_1,s_1+a]\times \cdots \times [s_k,s_k+a],
    \end{equation*}
    which clearly equals to the RHS of Equation \eqref{LHS} as $f_1$ is $\Rk$-equivariant.
    Now assume $\tt\in\Rk\setminus\bigcup\limits_{\mathbf s\in H(\rr x)}[s_1,s_1+a]\times \cdots \times [s_k,s_k+a]$. Thus by definition $g_1(\rr x)(\tt)=f_1(\rr x)(\tt)$. As $f_1$ is $\Rk$-equivariant we have  $g_1(\rr x)(\tt)=f_1(x)(\rr+\tt)$. We claim $g_1(x)(\rr+\tt)=f_1(x)(\rr+\tt)$. Indeed it is enough to show $\rr+\tt\in \Rk\setminus\bigcup\limits_{\mathbf s\in H( x)}[s_1,s_1+a]\times \cdots \times [s_k,s_k+a]$ which follows easily from the condition on $\tt$.  Combining the two cases, we conclude that $g_1$ is $\Rk$-equivariant. Furthermore for all $x\in X$ and $\ss\in H(x)$ and $\tt\in[s_1,s_1+a]\times \cdots \times [s_k,s_k+a]$ it holds
   \begin{equation}
        \begin{split}
    g_1(x)(\tt)&=f_1(x)(\tt)+q(\ss x)\cdot u(\ss x)(\mathbf {t-s})=f_1(x)(\tt)+q(\ss x)\cdot (g(\ss x)(\mathbf {t-s})-f_1(\ss x)(\mathbf {t-s}))\\
    &=(1-q(\ss x))f_1(x)(\tt)+q(\ss x)\cdot g(\ss x)(\mathbf {t-s})\in [0,1].
    \end{split}
     \end{equation}
     We conclude that for all $x\in X$, $g_1(x)(\R^k)\subset [0,1]$.
          Finally observe that for all $x\in X$ and $\tt\in\Rk$
    \begin{equation}\label{eq:g_1-f_1}
        |g_1(x)(\tt)-f_1(x)(\tt)|\leq \|q_{|S}\|_{\infty} \|(g-f_1)_{|X\times [0,a]^k}\|_{\infty}\leq 1\cdot 2\delta
    \end{equation}
      Therefore for all $x\in X$ and $\mathbf t\in\Rk$,
    \begin{equation}\label{eq:g_1-f}
        \begin{split}
            |g_1(x)(\tt)-f(x)(\tt)|&\leq |g_1(x)(\tt)-f_1(x)(\tt)|+|f_1(x)(\tt)-f(x)(\tt)|\\
            &=|g_1(x)(\tt)-f_1(x)(\tt)|+|(1-\delta)f(x)(\tt)+\delta f_0(x)(\tt)-f(x)(\tt)|\\
            &\stackrel{Eq. \eqref{eq:g_1-f_1}}{\leq}2\delta+\delta|f_0(x)(\tt)-f(x)(\tt)|\leq 4\delta.
        \end{split}
    \end{equation}
    \end{proof}
    We proceed with the proof of Theorem \ref{Rkdense1}. If $x\in A$, there exists $\mathbf s\in[-a,a]^k$ with $\ss x\in A_0\subset S$. Therefore $q(\ss x)=1$ and $\mathbf 0\in H(\ss x)$. By the definition of $g_1(\ss x)$, for $\mathbf t\in[0,a]^k$
    \begin{equation*}\label{g1=g}
        g_1(x)(\mathbf {s+t})=g_1(\ss x)(\tt)=f_1(\ss x)(\tt)+u(\ss x)(\mathbf {t-0})=g(\ss x)(\tt).
    \end{equation*}
    By the third property of the function $g(x)$, we obtain that $g_1(x):\Rk\to[0,1]$ is a non-constant function for each $x\in A$, i.e. $g_1(A)\cap\Fix(\Lip_{1}(\Rk))=\emptyset$. We conclude $g_1\in G(A)$. Since $f$ and $\delta$ were chosen arbitrarily, $G(A)$ is dense in $C_{\Rk,h_0}(X;\Lip_{1}(\Rk))$.
\end{proof}

\begin{thm}\label{Rkdense2}(cf. \cite[Lemma 3.4]{GJT19})
Let $(X,\Rk)$ be a weakly locally free topological multidimensional flow and  $h_0:\Fix(X,\Rk)\hookrightarrow[0,1]$ be a topological  embedding. Then for any two distinct points $p,q\in X\setminus\Fix(X,\Rk)$, there exist closed neighborhoods $B$ and $C$ of $p$ and $q$ in $X$ respectively such that $G(B,C)$\footnote{See Equation \eqref{G(B,C)}.}
is open and dense in $C_{\Rk,h_0}(X;\Lip_{1}(\Rk))$.
\end{thm}
\begin{proof}
    Fix two distinct points $p,q\in X\setminus\Fix(X,\Rk)$. By Theorem \ref{thm:local section}, there are two $\Rk$-local sections $(a,S_1)$ and $(a,S_2)$ around $p$ and $q$ respectively. We may assume that $[-a,a]^k\cdot S_1$ and $[-a,a]^k\cdot S_2$ are disjoint. 
    Similarly to what is done in the proof of Theorem \ref{Rkdense1},  for each $x\in X$, we define $H(x):=\{\mathbf t\in\Rk|\,\, \tt x\in S_1\cup S_2\}$. 
    We choose closed neighborhoods $B_0$ of $p$ in $S_1$ and $C_0$ of $q$ in $S_2$ satisfying $B_0\subset\Int([-a,a]^k\cdot S_1)$ and $C_0\subset\Int([-a,a]^k\cdot S_2)$. 
    Fix a continuous function $\widetilde{q}:X\to[0,1]$ such that $\widetilde{q}=1$ on $B_0\cup C_0$ and $\supp(\widetilde{q})\subset\Int([-a,a]^k\cdot S_1)\cup\Int([-a,a]^k\cdot S_2)$.
    We define closed neighborhoods $B$ and $C$ of $p$ and $q$ respectively by
    \begin{equation*}
        B=\bigcup\limits_{\mathbf t\in[-\frac{a}{16},\frac{a}{16}]^k}\tt B_0
        ,~C=\bigcup\limits_{\mathbf t\in[-\frac{a}{16},\frac{a}{16}]^k}\tt C_0.
    \end{equation*}
     
 \noindent   It is obvious that the set $G(B,C)$ defined by Equation \eqref{G(B,C)} is open in $C_{\Rk,h_0}(X;\Lip_{1}(\Rk))$. We now prove that $G(B,C)$ is also dense in $C_{\Rk,h_0}(X;\Lip_{1}(\Rk))$. Fix $f\in C_{\Rk,h_0}(X;\Lip_{1}(\Rk))$ and $0<\delta<1$. We assume w.l.o.g.\ that 
    \begin{equation*}
        \delta<d(B_0,C_0):=\min\limits_{x\in B_0,y\in C_0}d(x,y).
    \end{equation*}
    Let $f_0$ be given by Lemma \ref{GJTLemma3.1}. We define $f_1\in C_{\Rk,h_0}(X;\Lip_{1}(\Rk))$ by Equation \eqref{delta perturb}. By the definition it is clear that $\tau:=\Lip(f_1)\leq1-\frac{\delta}{2}$ and $\max\limits_{x\in X}\max\limits_{\mathbf t\in\Rk}|f_1(x)(\tt)-f(x)(\tt)|<2\delta$. We apply Lemma \ref{main lemma} from Subsection \ref{subsec:main_lemma} below to the map
    \begin{equation*}
        X\ni x\mapsto f_1(x)|_{[0,a]^k}\in\Lip_{\tau}([0,a]^k)
    \end{equation*}
    and find $\AA \subset (0,a)^k$ a finite non-empty subset and $g\in C(X;\Lip_{1}([0,a]^k))$ such that
    \begin{enumerate}[$(1)$]
    \item  $|g(x)(\tt)-f_1(x)(\tt)|<2\delta$ for all $x\in X$ and $\mathbf t\in[0,a]^k$.
    \item  For all $x\in X$, $g(x)(\tt)=f_1(x)(\tt)$ if $\mathbf t\in\Edge([0,a]^k)$.
    \item If $x,y\in X$ and $\mathbf w\in[-\frac{a}{8},\frac{a}{8}]^k$ satisfy that for all $\tt \in\left[\frac{a}{8},\frac{7a}{8}\right]^k\cap \AA$  it holds 
   \begin{equation*}
       g(x)(\mathbf{t+w})=g(y)(\tt)
   \end{equation*}
   then $\mathbf w=\mathbf{0}$ and $d(x,y)<\delta$.
    \end{enumerate}
   Set $u(x)(\tt)=g(x)(\tt)-f_1(x)(\tt)$ for $x\in X$ and $\mathbf t\in[0,a]^k$. We define $g_1:X\to\Lip_{1}(\Rk)$ as follows. Fix $x\in X$,
    \begin{itemize}
        \item For each $\mathbf s\in H(x)$ and $\mathbf t\in[s_1,s_1+a]\times \cdots \times [s_k,s_k+a]$, set
        \begin{equation*}
            g_1(x)(\tt)=f_1(x)(\tt)+\widetilde{q}(\ss x)\cdot u(\ss x)(\tt-\mathbf s).
        \end{equation*}
        \item For $\mathbf t\in\Rk\setminus\bigcup\limits_{\mathbf s\in H(x)}[s_1,s_1+a]\times \cdots \times [s_k,s_k+a]$, set $g_1(x)(\tt)=f_1(x)(\tt)$.
    \end{itemize}
    As in Claim \ref{equi and perturb}, we may prove that $g_1\in C_{\Rk,h_0}(X;\Lip_{1}(\Rk))$  and $\max\limits_{x\in X}\max\limits_{\mathbf t\in\Rk}|g_1(x)(\tt)-f(x)(\tt)|\leq 4\delta$.  We now prove that $g_1(B)\cap g_1(C)=\emptyset$. Suppose that $x\in B$ and $y\in C$ satisfy $g_1(x)=g_1(y)$. Thus there exist $\mathbf {s, r}\in[-\frac{a}{16},\frac{a}{16}]^k$ such that $\ss x\in B_0,\rr y\in C_0$. Moreover, $\widetilde{q}(\ss x)=\widetilde{q}(\rr y)=1$ and $\mathbf 0\in H(\ss x)\cap H(\rr y)$.
    By the definition of $g_1$, we obtain that  
    \begin{equation*}
        g_1(x)(\mathbf {s+t})=g(\ss x)(\tt)~~\text{and}~~g_1(y)(\mathbf {r+t})=g(\rr y)(\tt)~\text{for}~\mathbf t\in[0,a]^k.
    \end{equation*}
       Set $\mathbf w=\mathbf {r-s}\in[-\frac{a}{8},\frac{a}{8}]^k$. We conclude that for $\tt\in \left[\frac{a}{8},\frac{7a}{8}\right]^k\cap \AA$
   \begin{equation}\label{eq:grry}
       g(\rr y)(\tt)=g_1(y)(\mathbf {r+t})=g_1(y)(\mathbf {s+t+r-s}) =g_1(y)(\mathbf {s+t+w});\,  g(\ss x)(\mathbf {t+w})= g_1(x)(\mathbf {s+t+w}).
           \end{equation}
 It follows from $g_1(x)=g_1(y)$ that $g_1(y)(\mathbf {s+t+w})=g_1(x)(\mathbf {s+t+w})$. Thus from Equation \eqref{eq:grry}, we conclude that for $\tt\in \left[\frac{a}{8},\frac{7a}{8}\right]^k\cap \AA$, $g(\ss x)(\mathbf {t+w})=g(\rr y)(\tt)$. By the third property of $g$, it holds $d(\ss x,\rr y)<\delta$, which contradicts the fact that $\delta<d(B_0,C_0)\leq d(\ss x,\rr y)$. In conclusion, $g_1(B)\cap g_1(C)=\emptyset$ and therefore $g_1\in G(B,C)$. Since $f$ and $\delta$ were chosen arbitrarily, $G(B,C)$ is dense in $C_{\Rk,h_0}(X;\Lip_{1}(\Rk))$.
\end{proof}

\subsection{Main lemma}\label{subsec:main_lemma}

\begin{lem}\label{main lemma}(cf. \cite[Lemmas 2.1 \& 2.5]{GJT19})
    Let $f\in C(X;\Lip_{\tau}([0,a]^k))$ for some $0<\tau<1$ and $a>0$. 
      Then for any $\delta>0$, there exists $\AA \subset (0,a)^k$ a finite non-empty subset and $g\in C(X;\Lip_{1}([0,a]^k))$ such that
   \begin{enumerate}[$(1)$]
   \item  $\max\limits_{x\in X}\max\limits_{\mathbf t\in[0,a]^k}|g(x)(\tt)-f(x)(\tt)|<2\delta$.
   \item For all $x\in X$ and $\tt\in\Edge([0,a]^k)$,  $g(x)(\tt)=f(x)(\tt)$.
     \item $g(x)_{|\AA}$ is a non-constant function for all $x\in X$.
        \item If $x,y\in X$ and $\mathbf w\in[-\frac{a}{8},\frac{a}{8}]^k$ satisfy that for all $\tt \in\left[\frac{a}{8},\frac{7a}{8}\right]^k\cap \AA$\footnote{From the proof it follows $\left[\frac{a}{8},\frac{7a}{8}\right]^k\cap \AA\neq\emptyset$.}  it holds 
   \begin{equation*}
       g(x)(\mathbf{t+w})=g(y)(\tt)
   \end{equation*}
   then $\mathbf w=\mathbf{0}$ and $d(x,y)<\delta$.
   \end{enumerate}
\end{lem}
\begin{proof}
 We start by an outline of the proof. 
The cube $[0,a]^k$ is decomposed into the outer edge $\Edge([0,a]^k)$, and a finite evenly spaced linear grid $\AA$. We also connect by an interval $\LL$ a certain subset of (collinear) points in $\AA$.
 We then perturb $f(x)$ on $\AA$ and $\LL$ in  a specific way aiming  to achieve Properties (3) and (4).  The difficulty lies in changing the values of $f$ outside $\AA$ and $\LL$ without adjusting the values on $\AA$ and $\LL$, however using appropriate Lipschitz function extension theorems this can be done, simultaneously fulfilling Properties (1) and (2).   

    Fix $0<b<c<a$ with $b=a-c<\frac{a}{8}$. Choose an open cover $\{U_1,\ldots,U_M\}$ of $X$ satisfying
    \begin{equation}\label{cover}
        \text{diam}f(U_m)<\frac{\delta}{8},~\forall 1\leq m\leq M.
    \end{equation}
    Fix a point $p_m\in U_m$ for each $m$. Let $N\geq 3$ be a natural number and set
    \begin{equation*}
        \Delta:=\frac{c-b}{N-1}.
    \end{equation*}
 We divide the interval $[b,c]$ into $(N-1)$ intervals of length $\Delta$\label{notation_Delta}:
 \begin{equation*}\label{notation_a_n}
     b=a_1<a_2<\ldots<a_N=c,~~a_{n+1}-a_n=\Delta,  ~~\forall 1\leq n \leq N-1.
 \end{equation*}
 Denote by $a_0:=0, a_{N+1}:=a$ and
     $\vec{\frac{a}{2}}=(\frac{a}{2},\ldots, \frac{a}{2})\in \R^{k-1}$. Set $\mathbf{A}:=\{(\vec{\frac{a}{2}}, a_{r}): 1\leq r\leq N\}$. As $N\geq 3$, $a_2\leq \frac{a}{2}$. Recall $a_1< \frac{a}{8}$.
 Denote by $L,Q$ the unique integers $1\leq L,Q\leq N-1$ such that
 \begin{equation*}\label{notation_L,Q}
     a_L\leq \frac{a}{4}<a_{L+1}\leq a_Q\leq\frac{a}{2}<a_{Q+1}.
 \end{equation*}
Set $\LL:=\mathbf{A}\cap(\{\vec{\frac{a}{2}}\}\times[a_{L+1},\frac{a}{2}])=\{(\vec{\frac{a}{2}}, a_n): L+1\leq n\leq Q\}$.
 Note $\LL\subset \Int([\frac{a}{8}, \frac{7a}{8}]^k)$.
 We define vectors $\mathbf e\in\mathbb{R}^{\mathbf{A}}$ and $\widetilde{\mathbf e}\in\mathbb{R}^{\LL}$ by 
    \begin{equation*}\label{notation_e}
      \mathbf e:=\{\underbrace{1,\ldots,1}_{N}\} \text{ and } \widetilde{\mathbf e}:=\{\underbrace{1,\ldots,1}_{Q-L}\}
    \end{equation*}
    respectively. We choose $N$ large enough such that
    \begin{equation} \label{eq:Delta}
        \Delta<\frac{\delta}{8},\,Q-L\geq 2M.
    \end{equation}
      Since $|\mathbf{A}|=N> Q-L\geq 2M\geq M+1$,
    using Lemma \ref{lem:ap_1}, Lemma \ref{lem:ap_2} and Lemma \ref{lem:Almost sure linear independence}, we may choose $|\mathbf{A}|$-vectors
        $\mathbf u_1,\ldots,\mathbf u_M\in (0,1)^{\mathbf{A}}$
    such that
    \begin{enumerate}[$(1)$]
    \item  $|f(p_m)(\aa)-\mathbf u_m(\aa)|<\eta$ for all $1\leq m\leq M$ and $\aa\in \AA$.
    \item The $(M+1)$ vectors $\mathbf e,\mathbf u_1,\ldots,\mathbf u_M$ in $(0,1)
    ^{\mathbf{A}}$ are linearly independent\footnote{Here one uses $|\AA|\geq M+1$.}. 
     \item  
     For $1\leq m\leq M$ define 
     \begin{equation}\label{eq:D_QL}
    D_{Q-L}\mathbf u_m=(\mathbf u_m(\vec{\frac{a}{2}},a_{L+2})-\mathbf u_m(\vec{\frac{a}{2}},a_{L+1}),\ldots,\mathbf u_m(\vec{\frac{a}{2}},a_{Q+1})-\mathbf u_m(\vec{\frac{a}{2}},a_Q))\in\mathbb{R}^{Q-L}.     
     \end{equation}
         The $(M+1)$ vectors $\widetilde{\mathbf e},D_{Q-L} \mathbf u_1,\ldots,D_{Q-L}\mathbf u_M$ in $\mathbb{R}^{Q-L}$ are linearly independent\footnote{Here one uses $Q-L\geq M+1$.}.
    \item  For any $\mathbf w\in\Rk$ with $\mathbf w+\LL\subset \mathbf{A}$, 
    \begin{equation*}
        \mathbf u_1|_{\LL},\mathbf u_2|_{\LL},\ldots,\mathbf u_M|_{\LL},\mathbf u_1|_{\mathbf w+\LL},\mathbf u_2|_{\mathbf w+\LL},\ldots,\mathbf u_M|_{\mathbf w+\LL}
    \end{equation*}
    are linearly independent in $\mathbb{R}^{\LL}$\footnote{Here one uses $|\LL|=Q-L\geq  2M$.}.
    \end{enumerate}
 Let $\{h_m\}_{m=1}^{M}$ be a partition of unity on $X$ such that $\supp(h_m)\subset U_m$ for each $m$. 
        For $x\in X$ we define $g_0(x):\Edge([0,a]^k)\cup\mathbf{A}\to[0,1]$ as follows.     
        Set $g_0(x)(\tt):=f(x)(\tt)$ if $\mathbf t\in\Edge([0,a]^k)$ and 
        \begin{equation}\label{eq:g_0 on A}
            g_0(x)(\mathbf a):=\sum\limits_{m=1}^{M}h_{m}(x)\mathbf u_m(\mathbf a),\,\,\,\aa \in \AA .
        \end{equation}
  Note that $g_0(x)(\AA)\subset (0,1)$ for all $x\in X$. On the interval from the point $(\vec{\frac{a}{2}}, a_{Q})$ to the point  $(\vec{\frac{a}{2}}, a_{L+1})$,  which corresponds to the convex closure of $\LL$, $\con(\LL)$, extend $g_0(x)$ to a piecewise linear function. That is, for every $\tt\in\con(\LL)\setminus\LL$, there exist unique $L+1\leq n<Q$ and $\mu\in(0,1)$ such that  $\tt=(\vec{\frac{a}{2}},(1-\mu)a_n+\mu a_{n+1})$. Then we define
\begin{equation*}
    g_0(x)(\tt)=(1-\mu)g_0(x)(\vec{\frac{a}{2}},a_{n})+\mu g_0(x)(\vec{\frac{a}{2}},a_{n+1}).
\end{equation*}
 Thus $g_0$ is defined on $\DD:=\Edge([0,a]^k)\cup \mathbf{A}\cup \con(\LL)$. Assume $N=N(f)$ is  big enough so that  $\Delta=\Delta(N,c,b)$ is small enough for Claim \ref{claim:g_0_Lip}. In addition assume  $0<\eta=\eta(\tau)<\frac{\delta}{4}$ as well as $\eta=\eta(\tau)$ is small enough for Claim \ref{claim:g_0_Lip}. By Claim \ref{claim:g_0_distance} and Claim \ref{claim:g_0_Lip} below, $g_0\in C(X;\Lip_{\tau'}(\DD))$ for some $\tau\leq\tau'<1$ and for all $x\in X$ and $\|g_0(x)-f(x)_{|\DD}\|_\infty<\frac{\delta}{2}$.
     \item
     We now use Theorem \ref{thm:Lip_Ext}, the McShane--Whitney extension theorem (which was prefigured by a result of Banach, see \cite[Theorem IV.7.5]{Ban51}; for a full historical account see \cite{gutev2020lipschitz}),  to extend $g_0$ to a function $g_1\in C(X;\Lip_{\tau'}([0,a]^k,\R))$ given
          by the following explicit formula
     \begin{equation*}\label{notation_g1}
        g_1(x)(\tt)=\max\limits_{\mathbf s\in \DD}\big(g_0(x)(\ss)-\tau'\|\mathbf{s-t}\|_2\big).
    \end{equation*}
    \item
Define  $g_2(x):=\min\{g_1(x),f(x)+\frac{\delta}{2}\}$ for $x\in X$. As $\tau'\geq \tau$ by Lemma \ref{lem:max Lip}, $g_2(x)$ is a $\tau'$-Lipschitz function so that $g_2(x)< f(x)+\delta$ and  $(g_2(x))_{|\DD}=(g_1(x))_{|\DD}=(g_0(x))_{|\DD}$. 
Now define $g_3(x):=\max\{g_2(x),f(x)-\frac{\delta}{2}\}$. Again by Lemma \ref{lem:max Lip}, $g_3(x)$ is a $\tau'$-Lipschitz function so that $f(x)-\delta<g_3(x)<f(x)+\delta$ and  $(g_3(x))_{|\DD}=(g_0(x))_{|\DD}$. 
    \item
Fix $x\in X$. By Theorem \ref{thm:Lip_C_infty} applied to $g_3(x):\overline{\Omega}\rightarrow \R$ where $\Omega=(0,a)^k\setminus \DD$ and $\overline{\Omega}=[0,a]^k$, one may find  $1>\tau''>\tau'$, such that the for every $x\in X$ the function  $g_4(x):[0,a]^k\rightarrow \R$ given by the following formula for $\bb\in \Omega$
\begin{equation} \label{eq:g_4}
g_4(x)(\bb)=\int_{\overline{B}} \theta(\tt)g_3(x)(\bb-\rho(\bb)\tt) d\tt
\end{equation}
where $\overline{B}$ is the closed unit ball in $\Rk$ and $\theta, \rho\in C^\infty(\Rk)$, have the following properties for all $x\in X$: (1) $g_4(x)$ is $\tau''$-Lipschitz ; (2)   $g_4(x)\in C^\infty( (0,a)^k\setminus \DD, \R)$; (3) $f(x)-2\delta<g_4(x)< f(x)+2\delta$ and  $(g_4(x))_{|\DD}=(g_3(x))_{|\DD}=(g_0(x))_{|\DD}$. The last property together with Equation \eqref{eq:g_4} imply that  $g_4(x)$ depends continuously on $x$. Define $$g(x):=\max\{0,\min\{1,g_4(x)\}\}.$$
By Lemma \ref{lem:max Lip},   $g\in C(X;\Lip_1([0,a]^k))$. As for all $x\in X$, $f(x)([0,a]^k)\subset [0,1]$,   $\|g(x)-f(x)\|_\infty<2\delta$ which verifies the first property of $g$.
As for all $x\in X$, $g_0(x)(\DD)\subset [0,1]$, $g(x)_{|\DD}=(g_4(x))_{|\DD}=(g_0(x))_{|\DD}$. In particular  $g(x)(\tt)=f(x)(\tt)$ ($\tt\in\Edge([0,a]^k)$) for all $x\in X$ which verifies the second property of $g$.

We proceed to the proof of the third property of $g$ in the statement.
       For each $x\in X$, the function $g(x):[0,a]^k\to[0,1]$ is a non-constant function because $g(x)|_{\mathbf{A}}=g_0(x)|_{\mathbf{A}}=\sum\limits_{m=1}^{M}h_{m}(x)\mathbf u_m\notin\mathbb{R}\mathbf e$ by Property $(2)$ of $\mathbf u_m$. 
    
Finally, we verify the forth property of $g$.
Suppose that there exist $x,y\in X$ and $\mathbf w\in[-\frac{a}{8},\frac{a}{8}]^k$ such that $g(x)(\mathbf {t+w})=g(y)(\tt)$ for all $\tt\in[\frac{a}{8},\frac{7a}{8}]^k\cap \AA$. 
First, we shall prove that $\mathbf w+\LL\subset \mathbf{A}$. Note that as $\LL\subset [\frac{a}{4},\frac{a}{2}]^k$, $\ww+\LL\subset [\frac{a}{8},\frac{5a}{8}]^k$. Therefore if there exists $\lambda\in \LL$, so that $\ww+\ll\in \AA$, then $\ww+\LL\subset \AA$.   
Thus if $\mathbf w+\LL\subset \mathbf{A}$ does not hold then $(\mathbf w+\LL)\cap \mathbf{A}=\emptyset$. Assume for a contradiction the latter. Note that by construction $\LL\subset [\frac{a}{8},\frac{7a}{8}]^k$ as $a_L\leq \frac{a}{4}<a_{L+1}\leq a_Q\leq\frac{a}{2}<a_{Q+1}$. 
Thus for all $\mathbf t=(\vec{\frac{a}{2}}, a_n)\in\LL$ with $L+1\leq n\leq Q-1$, $(0,1)\ni g_4(y)(\tt)=g(y)(\tt)=g(x)(\tt+\mathbf {w})$. As $\mathbf {t+w}\in (0,a)^k\setminus\mathbf{A}$ there are two possibilities. Either $\mathbf {t+w}\in (0,a)^k\setminus\mathbf{D}$ or $\mathbf {t+w}\in \con(\LL)\setminus\mathbf{A}$. In the first case $g_4(x)(\tt+\mathbf {w})=g(x)(\tt+\mathbf {w})\in (0,1)$, so $g(y)$ is differentiable at $\mathbf t$ (as it equals $g_4(y)$ in a neighborhood of $\tt$). In the second case\footnote{Recall $g(y)$ is piecewise linear on $\con(\LL)$.} $\tt$ is not a discontinuity point for the derivative of $g(y)$ in the direction $t_k$ (the last coordinate). Thus in both case we may conclude  $\tt$ is not a discontinuity point for the derivative of $g(y)$ in the direction $t_k$  at  $\mathbf t$ and therefore the left hand derivative and right hand derivative in the direction $t_k$  must equal\footnote{As $L+1\leq n\leq Q-1$, $\tt$ is not a boundary point of the interval $\con(\LL)$.}:
    
    \begin{equation*}
        g(y)(\vec{\frac{a}{2}},a_{n+1})-g(y)(\vec{\frac{a}{2}},a_n)=g(y)(\vec{\frac{a}{2}},a_{n+2})-g(y)(\vec{\frac{a}{2}},a_{n+1}),~\forall L+1\leq n\leq Q-1,
    \end{equation*}
    and therefore using the definition of $g(y)$ at $\AA$
    \begin{equation*}
    \begin{split}
        \sum\limits_{m=1}^{M}h_{m}(y)(\mathbf u_m(\vec{\frac{a}{2}},a_{n+1})-\mathbf u_m(\vec{\frac{a}{2}},a_n))=\sum\limits_{m=1}^{M}h_{m}(y)(\mathbf u_m(\vec{\frac{a}{2}},a_{n+2})&-\mathbf u_m(\vec{\frac{a}{2}},a_{n+1})),\\
        &~\forall L+1\leq n\leq Q-1.
    \end{split}
    \end{equation*}
    Recall  the definition of $D_{Q-L}$ in Equation \eqref{eq:D_QL}.
   Denoting $D_{Q-L}\mathbf u_m=\big((D_{Q-L}\mathbf u_m)_1,\ldots,  (D_{Q-L}\mathbf u_m)_{Q-L} \big)$ one may write the equations above as 
    $$
\sum\limits_{m=1}^{M}h_{m}(y)(D_{Q-L}\mathbf u_m)_1=\sum\limits_{m=1}^{M}h_{m}(y)(D_{Q-L}\mathbf u_m)_2=\cdots=\sum\limits_{m=1}^{M}h_{m}(y)(D_{Q-L}\mathbf u_m)_{Q-L},
    $$
    equivalently $\sum\limits_{m=1}^{M}h_{m}(y)D_{Q-L}\mathbf u_m\in\mathbb{R}\widetilde{\mathbf e}$ (recall $\widetilde{\mathbf e}=\{\underbrace{1,\ldots,1}_{Q-L}\}$), which contradicts  Property $(3)$ of $\mathbf u_m$. Thus $\mathbf w+\LL\subset \mathbf{A}$. 
        Since $g(x)(\mathbf {t+w})=g(y)(\tt), \forall \mathbf t\in[\frac{a}{8},\frac{7a}{8}]^k\cap \AA$, it holds by Equation \eqref{eq:g_0 on A}
    \begin{equation*}
        \sum\limits_{m=1}^{M}h_{m}(x)\mathbf u_m|_{\mathbf w+\LL}=\sum\limits_{m=1}^{M}h_{m}(y)\mathbf u_m|_{\LL}.
    \end{equation*}
    Then by Property $(4)$ of $u_m$, it holds $\mathbf w=\mathbf{0}$ and $h_{m}(x)=h_{m}(y)$ for each $1\leq m\leq M$. Therefore $x,y\in U_m$ for some $m$ and thus $d(x,y)\leq\text{diam}(U_m)<\delta$. Q.E.D.
    
\end{proof}

\begin{claim}\label{claim:g_0_distance}
If $\eta<\frac{\delta}{4}$, then for all $x\in X$, $\|g_0(x)-f(x)_{|\DD}\|_\infty<\frac{\delta}{2}$.
\end{claim}
\begin{proof}
   
   Fix $\tt\in \mathbf{A}\cup \con(\LL)$. Then there exist unique $\aa, \bb\in \AA$ with $\|\aa-\bb\|_2=\Delta$ and $0\leq\mu\leq 1$ so that $\tt=\mu \aa +(1-\mu)\bb$. For $x\in X$ it holds by property (1) of $\mathbf u_m$, the assumption $\eta<\frac{\delta}{4}$,   the fact that $f(x)$ is $\tau$-Lipschitz, $\diam f(U_m)<\frac{\delta}{8}$ (see Equation \eqref{cover}) and $\Delta< \frac{\delta}{8}$ (see Equation \eqref{eq:Delta}):

        \begin{equation*}
            \begin{split}
            &|\mathbf u_m(\mathbf{a})-f(x)(\mathbf{t})|\leq|\mathbf u_m(\mathbf{a})-f(p_m)(\mathbf{a})|+|f(p_m)(\mathbf{a})-f(p_m)(\mathbf{t})|+|f(p_m)(\mathbf{t})-f(x)(\mathbf{t})|\\
            &\leq \eta +\tau(1-\mu)\Delta+\frac{\delta}{8}< \frac{\delta}{4}+\frac{\delta}{8}+\frac{\delta}{8}=\frac{\delta}{2}.
            \end{split}
            \end{equation*}
         \noindent   
        Similarly  $|\mathbf u_m(\mathbf{b})-f(x)(\mathbf{t})|<\frac{\delta}{2}$. Thus for $x\in X$ one may bound $|g_0(x)(\tt)-f(x)(\mathbf{t})|=|\mu g_0(x)(\mathbf{a})+ (1-\mu) g_0(x)(\mathbf{b})-f(x)(\mathbf{t})|$ from above by 
        \begin{equation*}
            \begin{split}
                &\mu\sum\limits_{m=1}^{M}h_{m}(x)|\mathbf u_m(\mathbf{a})-f(x)(\mathbf{t})|+(1-\mu)\sum\limits_{m=1}^{M}h_{m}(x)|\mathbf u_m(\mathbf{b})-f(x)(\mathbf{t})|\\
                &<\mu\frac{\delta}{2}+(1-\mu)\frac{\delta}{2}=\frac{\delta}{2}.
            \end{split}
        \end{equation*}
 We conclude $\|g_0(x)-f(x)_{|\DD}\|_\infty<\frac{\delta}{2}$ for all $x\in X$. 
\end{proof}

\begin{claim}\label{claim:g_0_Lip}
If $\eta=\eta(\tau)>0$ and $\Delta=\Delta(N,b,c)>0$ are small enough, then  $g_0\in C(X;\Lip_{\tau'}(\DD))$ for some $\tau\leq\tau'<1$.
\end{claim}

\begin{proof}
It is clear $g_0$ is continuous in $x\in X$ and for all $x\in X$, $g_0(x)$ is a function taking values in $[0,1]$. Recall $\eta$ appears only in property (1) of $\mathbf u_m$.
Choosing $\eta=\eta(\tau)>0$  small enough and using the fact that $d_{\|\cdot\|_2}(\AA,\Edge([0,a]^k))>0$ one concludes  that $(g_0(x))_{|\Edge([0,a]^k)\cup \mathbf{A}}$ is $\tau'$-Lipschitz for some $1>\tau'\geq \tau$ for all $x\in X$.   By the piecewise linear construction it is clear that $(g_0(x))_{|\AA\cup\con(\LL)}$ is $\tau'$-Lipschitz as the concatenation of two $\tau'$-Lipschitz functions is $\tau'$-Lipschitz\footnote{That is if $f: [a, b] \to \mathbb{R}$ and $g: [b, c] \to \mathbb{R}$ are both $\tau'$-Lipschitz so that $f(b) = g(b)$ then the concatenated function $h: [a, c] \to \mathbb{R}$ is $\tau'$-Lipschitz.}.
In order to show $g_0(x)$ is $\tau''$-Lipschitz on $\DD=\Edge([0,a]^k)\cup \mathbf{A}\cup \con(\LL)$ for some $1>\tau''\geq \tau'$, we denote $A=\AA$, $E=\Edge([0,a]^k)$ and $T=\con(\LL)$. We  need to show that the conclusion of 
 Lemma \ref{lem:Lipschitz constant continuity} holds for $g_0(x)$ for all $x\in X$ with some $\delta, \epsilon>0$ so that $\tau''=\frac{\tau'+\epsilon}{1-\delta}$.  Obviously the first condition holds w.r.t.\ $\tau'$. In the second condition one may take
 $$\delta=\frac{\Delta}{d_{\|\cdot\|_2}(\AA,\Edge([0,a]^k))},\, \epsilon=\frac{\Delta \tau'}{d_{\|\cdot\|_2}(\AA,\Edge([0,a]^k))}.
 $$
Thus in order to achieve $\delta,\epsilon>0$ small one makes $\Delta$ small enough by enlarging $N$ as needed.  
 
\end{proof}

\appendix
\renewcommand{\appendixname}{Appendix~\Alph{section}}
\section{Existence of local sections}
\indent\indent In order to prove Theorem \ref{thm:local section}, we will reduce it to a statement from  \cite{HM66}\footnote{Note Gleason's famous result in  \cite{Gle50} concerns  local cross-sections in the case of compact Lie group actions so it does not apply to $\R^k$-actions.}. As  Hofmann and Mostert  worked in quite a general setting, one has to show it applies in the setting of this article. For simplicity we adapt the definitions  of \cite[Appendix II]{HM66} to the abelian Lie action setting.   

\begin{defn}\cite[Appendix II, Definition 1.1]{HM66}\label{local space}
    Let $G$ be an abelian Lie group. A \textbf{local $G$-space} $(X,G)_{\textrm{loc}}$ is a pair consisting of a completely regular space $X$ and a continuous map $A:D(A)\to X$ whose domain $D(A)$ is an open neighborhood of
     $\{\mathbf{0}\} \times X $ in $ G \times X$
    which satisfies the following two conditions:
    \begin{enumerate}[$(a)$]
    
\item $A(\mathbf{0},x)=x$ for all $x\in X$.
    
    \item If $(g,x),A(h,A(g,x)),A(hg,x)\in D(A)$, then $A(hg,x)=A(h,A(g,x))$.
        \end{enumerate}

Two local $G$-spaces $(X,A)_{\textrm{loc}}$ and $(X,B)_{\textrm{loc}}$ are said to be \textbf{equivalent} if there exists $U$ an open neighborhood of $\{\mathbf{0}\} \times X$ such that $U\subset D(A)\cap D(B)$ and $A_{|U}=B_{|U}$. This gives rise to an equivalence relation whose classes are denoted by $[(X,A)_{\textrm{loc}}]$. 
 \end{defn}
If no confusion arises we write $gx$ instead of $A(g,x)$. Clearly any t.d.s.\ $(X,G)$ is a local $G$-space. 

\begin{defn}\cite[Appendix II, Definition 1.9]{HM66}\label{def:isotropy_group}
    Let $(X,A)_{\textrm{loc}}$ be a local $G$-space for an abelian Lie group $G$. For $x_0\in X$, let $G_{x_0}^{(X,A)_{\textrm{loc}}}$ be the intersection of all subgroups of $G$ (algebraically) generated by the set of all $g\in G$ with $(g,x_0)\in D(B)$ and $B(g,x_0)=x_0$, where $(X,B)_{\textrm{loc}}$ ranges through the members of $[(X,A)_{\textrm{loc}}]$. 
\end{defn}

\begin{prop}
    Let $G$ be an abelian Lie group. Let $(X,G)$ be a t.d.s.\ and  $x_0\in X$. Then $G_{x_0}^{(X,G)}=G_{x_0}^\circ$, the identity component of $G_{x_0}:=\{g\in G|\,\, gx_0=x_0\}$. 
\end{prop}

\begin{proof}
    As the definition of $G_{x_0}^{(X,G)}$ involves an intersection it is enough to consider subgroups of $G$ generated by the set of all $g\in V$ with $gx_0=x_0$ $g\in V$, where $V$  ranges over open subsets of an arbitrary fixed open neighborhood $V_0$ of $\mathbf{0}\in G$. It is well known that a connected topological group is generated by any open neighborhood of the unit element, so for every $V$ as above the group generated must contain $G_{x_0}^\circ$.  As $G$ is a Lie group the identity component of $G$ is open. We thus consider $V_0=G_{x_0}^\circ$  to conclude  $G_{x_0}^{(X,G)}\subset G_{x_0}^\circ$. Combining the last two facts it holds  $G_{x_0}^{(X,G)}= G_{x_0}^\circ$.
\end{proof}
\begin{defn}\label{def:local cross section}(\cite[Appendix II, Definitions 1.10 \& 1.5 for certain t.d.s.]{HM66}\footnote{This definition should not be confused with Definition \ref{def:local section}.} )
Let $G$ be an abelian Lie group. Let $(X,G)$ be a t.d.s.\ and $x\in X$.  A \textbf{local cross section to the local orbits
at $x$} is a triple $(C, K; U)$ where 
\begin{itemize}
    \item 
$K$ is a closed set containing  $\mathbf{0}\in G$, so that there exists $K'\subset  G_{x}^\circ$, a compact symmetric
neighborhood (w.r.t.\ $ G_{x}^\circ$) of  $\mathbf{0}\in G_{x}^\circ$ such that $K'K$ is a neighborhood of $\mathbf{0}\in G$ and such that the multiplication map $K'\times K\rightarrow K'K$, $(k',k)\mapsto k'k$ is a homeomorphism; 
\item 
 $U$ is a closed neighborhood of $x$ in $X$;
 \item 
 $C$ is a closed subset of $U$ containing $x$, such that the map 
\begin{equation}
        K\times C\to U,~~(g,x)\mapsto g x
    \end{equation}
    is a homeomorphism onto $U$.
\end{itemize}
 
\end{defn}

\begin{thm}\cite[Appendix II, Theorem 1.11 for certain t.d.s.]{HM66}\label{existence of local cross sections}
   Let $G$ be an abelian Lie group. Let $(X,G)$ be a t.d.s.\ and  $x_0\in X$. Suppose $V$ is a neighborhood of $x_0$ and $W$ a neighborhood of $\mathbf{0}$ in $G$ such that the connected component of $\mathbf{0}$ of $G_{x_0}^\circ \cap W$ equals  for each $y\in V$ the connected component of $\mathbf{0}$ of $G_y^\circ\cap W$. Then there exists a local cross section to the local orbits at $x_0$.
    
    \end{thm}
    
\begin{proof}[Proof of Theorem \ref{thm:local section}]
 We now apply Theorem \ref{existence of local cross sections} to prove Theorem \ref{thm:local section}. Denote $Y=X\setminus\Fix(X,\R^k)$. As $(X,\Rk)$ is weakly locally free, for each $y\in Y$ there exists an open neighborhood $U_y$ of $\mathbf{0}$ in $\Rk$ such that $\ss y\neq y$  for all $\mathbf s\in U_y\setminus\{\mathbf{0}\}$. Thus for all $y\in Y$, 
\begin{equation*}
    G_{y}^\circ=\{\mathbf{0}\}.
\end{equation*}
Fix  a point $p\in Y$, an open set $p\in V\subset Y$ and denote $W=\Rk$. Then for each $y\in V$ it holds
\begin{equation*}
    G_y^\circ\cap W=G_p^\circ\cap W=\{\mathbf{0}\}.
\end{equation*}
 By Theorem \ref{existence of local cross sections}, there exist  a local cross section to the local orbits at $p$. That is,  we may find $K,C,U$ as in Definition \ref{def:local cross section}. We claim this gives rise to a local section at $p$ as defined in Definition \ref{def:local section}. Indeed it is enough to show $K$ is a neighborhood of $\textbf{0}$ in $\Rk$. As $(\Rk)_p^\circ=\{\textbf{0}\}$ where $(\Rk)_p^\circ$ is the 
  the identity component of $\Rk_{p}:=\{ \rr \in \Rk|\,\, \rr p=p\}$, we conclude $K'=\{\textbf{0}\}$ and the result follows.  Q.E.D.
\end{proof}

\section{Three lemmas from linear algebra}
\indent\indent For $\mathbf u=(x_1,\ldots,x_{n+1})\in\R^{n+1}$ set
\begin{equation*}
    D\mathbf u=(x_2-x_1,x_3-x_2,\ldots,x_{n+1}-x_n)\in\R^n.
\end{equation*}
\begin{lem}\cite[Lemma 2.3]{GJT19}\label{lem:ap_1}
    Let $l\geq m+1$ and set $\mathbf e:=\{\underbrace{1,\ldots,1}_{l}\}\in\R^l$. Then the set of $(\mathbf u_1,\ldots,\mathbf u_m)\in\R^{l+1}\times\ldots\times\R^{l+1}=(\R^{l+1})^m$ such that the $m+1$ vectors $\mathbf e,D\mathbf u_1,D\mathbf u_2,\ldots,D\mathbf u_m\in \R^l$ are linearly independent is open and dense in $(\R^{l+1})^m$.
\end{lem}

\begin{lem}\cite[Lemma 2.4]{GJT19}\label{lem:ap_2}
    Let $n>l\geq 2m$. Then the set of $(\mathbf u_1,\ldots,\mathbf u_m)\in(\R^{n})^m$ such that, for any integer $\alpha$ with $2\leq\alpha\leq n-l+1$, 
    \begin{equation*}
        \text{the vectors } \mathbf u_1\vert_{1}^{l}, \mathbf u_2\vert_{1}^{l},\ldots, \mathbf u_m\vert_{1}^{l}, \mathbf u_1\vert_{\alpha}^{\alpha+l-1}, \mathbf u_2\vert_{\alpha}^{\alpha+l-1},\ldots, \mathbf u_m\vert_{\alpha}^{\alpha+l-1} \text{ are linearly independent in } \R^l
    \end{equation*}
    is open and dense in $(\R^{n})^m$. Here for $\mathbf u_i=(x_{i1},\ldots,x_{in})$
    \begin{equation*}
        \mathbf u_i\vert_{1}^{l}=(x_{i1},\ldots,x_{il}), \mathbf u_i\vert_{\alpha}^{\alpha+l-1}=(x_{i,\alpha},\ldots,x_{i,\alpha+l-1}).
    \end{equation*}
\end{lem}

\begin{lem}\cite[Lemma A.6]{Gut15Jaworski}
\label{lem:Almost sure linear independence}Let $m,s,r \in\N$. Let $V\subset\mathbb{R}^{m}$ be a linear subspace
with $\dim(V)=r$. If $r+s\leq m$, then almost surely w.r.t Lebesgue
measure for $(\vv_{r+1},\vv_{r+2},\ldots,\vv_{r+s})\in([0,1]^{m})^{s}$, $\dim(V+\spann(\vv_{r+1},\vv_{r+2},\ldots,\vv_{r+s}))=r+s$.
\end{lem}

\section{Some Facts about Lipschitz Functions}

\begin{lem}\label{lem:max Lip}
Let $f,g: [0,a]^k\rightarrow \R$ be $\tau$-Lipschitz for some $\tau>0$, then $\max\{f,g\}: [0,a]^k\rightarrow \R$ and $\min\{f,g\}: [0,a]^k\rightarrow \R$ are $\tau$-Lipschitz. 
\end{lem}
\begin{proof}
    Denote $h(\tt)=\max\{f(\tt),g(\tt)\}$. Note that for any $\ss\in [0,a]^k$ it holds $$f(\ss)\leq f(\tt)+|f(\ss)-f(\tt)|\leq h(\tt)+\tau\|\ss-\tt\|_2.$$
    Similarly $g(\ss)\leq  h(\tt)+\tau\|\ss-\tt\|_2.$
    Thus 
    $$h(\ss)\leq h(\tt)+\tau\|\ss-\tt\|_2.$$
    Reversing the roles of $\ss$ and $\tt$, one has $|h(\ss)-h(\tt)|\leq \tau\|\ss-\tt\|_2$ as desired. Now notice  $\min\{f,g\}=-\max\{-f,-g\}$ to conclude the same result for the minimum function.  
\end{proof}

\begin{lem}\label{lem:Lipschitz constant continuity}
Let $(M,d)$ be a metric space. Let $\tau',\epsilon>0$ and $0<\delta<1$. Let $A,E, T\subset M$ and $f:A\cup T\cup E\rightarrow \R$  so that
\begin{enumerate}
    \item 
    For all $a\in A$ and $c\in E$
    $$
\frac{|f(a)-f(c)|}{d(a,c)}\leq \tau'.
    $$
    \item 
    For all $t\in T$ there exists $a\in A$ so that 
   \begin{equation}\label{eq:a,b}
       \frac{d(a,t)}{d(A,E)}\leq \delta,\, \frac{|f(t)-f(a)|}{d(A,E)}\leq \epsilon.
   \end{equation}

    Then for all $t\in T$ and $c\in E$
    $$
\frac{|f(t)-f(c)|}{d(t,c)}\leq \frac{\tau'+\epsilon}{1-\delta}.
    $$
\end{enumerate}
\end{lem}
\begin{proof}
  Fix $t\in T$ and $c\in E$. Let $a\in A $ so that Equation \eqref{eq:a,b} holds.  We note $d(t,c)\geq |d(a,c)-d(t,a)|\geq d(a,c)(1-\frac{d(t,a)}{d(a,c)})\geq d(a,c)(1-\delta)$.
  Thus 
  $$
\frac{|f(t)-f(c)|}{d(t,c)}\leq \frac{|f(t)-f(a)|+|f(a)-f(c)|}{d(a,c)(1-\delta)}\leq \frac{\epsilon+\tau'}{1-\delta}. 
  $$
\end{proof}

\begin{lem}\label{lem:Lipschitz gradient}
Let  $\Omega\subset  \R^k$ be an open set and let $f: \Omega\rightarrow \R$ be $\tau$-Lipschitz for some $\tau>0$. Then for a.e. $\bb\in \Omega$, $f$ is differential at $\bb$ and $\|\nabla f(\bb)\|_2\leq \tau$.
\end{lem}
\begin{proof}
    By Rademacher’s Theorem (see e.g. \cite[\S 5.8, Theorem 6]{EvansPDE}), since $f$ is Lipschitz, it is differentiable almost everywhere in $\Omega$. Let $\bb\in \Omega$ be a point where $f$ is differentiable. Recall the directional derivative of $f$ at $\bb$ along a direction $\rr\in\R^k$ equals the inner product between $\rr$ and the gradient of $f$ at $\bb$. Let $\rr=\nabla f(\bb)$. It follows
    $$\langle \nabla f(\bb), \nabla f(\bb)\rangle=\lim_{t\rightarrow 0} \frac{f(\bb+t\nabla f(\bb))-f(\bb)}{t}\leq \lim_{t\rightarrow 0} \frac{1}{t}\tau \|t\nabla f(\bb)\|_2=\tau \|\nabla f(\bb)\|_2.$$
    We conclude $\|\nabla f(\bb)\|_2\leq \tau$.
\end{proof}

\begin{thm}(McShane--Whitney extension theorem (\cite{McSHANE1934ExtensionOR,whitney1934AnalyticEO}), see also \cite[Theorem 2.3]{gutev2020lipschitz})\label{thm:Lip_Ext}
     Let $(X,d)$ be a metric space, $A\subset  X$ and $\phi : A\rightarrow \R$ be a $\tau$-Lipschitz function. Then the function $\Phi : X \rightarrow \R\cup \{\infty\}$ given by the formula
     \begin{equation*}
        \Phi(x)=\sup\limits_{y\in A}\big(\phi(y)-\tau d(x,y)\big),
    \end{equation*}
     is  a $\tau$-Lipschitz function  which is an extension of $\phi$, i.e.\ so that $\Phi_{|A}=\phi$.
\end{thm}
\noindent
Let $\overline{B}$ be the closed unit ball in $\Rk$. A \textit{mollifier} is a nonnegative function $\theta \in C^\infty(\Rk,\R)$ supported in $\overline{B}$  such that $\int_{\Rk} \theta(\bb) \, d\bb = 1$.
\begin{thm}\label{thm:Lip_C_infty}(\cite[Proposition 4.1]{czarnecki2006approximation}; see also \cite[p. 887]{deville2019approximation}).
Let  $\Omega\subset  \R^k$ be an open and bounded set and $0<\tau<1$, $\delta, \epsilon>0$ given constants. There exist  a mollifier $\theta$ and a function $\rho\in C^\infty(\Rk)$ (depending only on $\Omega,\tau, \delta, \epsilon$) such that for any $\tau$-Lipschitz function $\phi : \overline{\Omega}\rightarrow \R$, there exists a $(\tau+\epsilon)$-Lipschitz function $\Phi: \overline{\Omega}\rightarrow \R$ such that $\Phi_{|\Omega}\in C^\infty(\Omega, \R)$, $\|\phi-\Phi\|_{\infty}<\delta$, $\Phi_{|\partial \Omega}=\phi_{|\partial \Omega}$ and\footnote{In particular for all $\bb\in \Omega$ and $\tt\in \overline{B}$, $ \bb-\rho(\bb)\tt\in \Omega$.}  such that for all  $\bb\in \Omega$ it holds
\begin{equation}\label{eq:Phi}
 \Phi(\bb)=\int_{\overline{B}} \theta(\tt)\phi(\bb-\rho(\bb)\tt) d\tt.
\end{equation}

\end{thm}

\begin{proof}
 We note that the  proof of the theorem is deduced from the proof of \cite[Proposition 4.1]{czarnecki2006approximation} which proves a  stronger statement for locally Lipschitz functions.  We give the main steps. Let $h: \Omega \rightarrow (0,\infty)$ be the distance function to the boundary $h(\bb):=d(\Rk \setminus \Omega, \bb)$.
By  \cite[Lemma 4.3]{cornet1999smooth} there exists a function $\rho\in C^{\infty}(\Omega,\R_{>0})$ such that for all $\bb\in \Omega$,
$
\rho(\bb) \leq \min \left\{\frac{\delta}{2\tau}, \frac{h(\bb)}{\tau}, \frac{h(\bb)}{2}\right\}
$ and $\|\nabla \rho(\bb)\|_2 \leq \frac{\epsilon}{\tau}$.
\noindent
As for all $\bb\in \Omega$, $0<\rho(\bb)\leq \frac{h(\bb)}{2}$, we conclude that for all $\bb\in \Omega$ and $\tt\in \overline{B}$, $ \bb-\rho(\bb)\tt\in \Omega$. Let $\theta$  be a mollifier. Define $\Phi$ on $\Omega$ by \eqref{eq:Phi}. In addition define  $\Phi_{|\partial \Omega}:=\phi_{|\partial \Omega}$. Note that $\Phi_{|\Omega}\in C^\infty(\Omega, \R)$ as easily seen by the change of variable $\yy=\bb-\rho(\bb)\tt$.  For all $\bb\in \Omega$, as $\rho(\bb)\leq \frac{\delta}{2\tau}$, it holds 
\[
|\Phi(\bb) - \phi(\bb)| = |\int_{\overline{B}} \theta(\tt)\big (\phi(\bb - \rho(\bb)\tt) - \phi(\bb)\big) \, d\tt|
\leq \int_{\overline{B}} \theta(\tt)\tau \|\rho(\bb)\tt\|_2 \, d\tt
\leq \int_{\overline{B}} \theta(\tt)\tau \rho(\bb) \, d\tt<\delta.
\]
\noindent
Similarly as $\rho(\bb)\leq \frac{h(\bb)}{\tau}$, it holds $|\Phi(\bb) - \phi(\bb)|\leq h(\bb)$ which implies $\Phi$ (which is defined on $\partial \Omega$ by $\Phi_{|\partial \Omega}=\phi_{|\partial \Omega}$) is continuous on $\overline{\Omega}$. We conclude $\|\phi-\Phi\|_{\infty}<\delta$. In order to prove $\Phi$ is  $(\tau+\epsilon)$-Lipschitz  in $\Omega$ we note that for all $\bb\in \Omega$
\begin{equation*}
\begin{split}
\|\nabla \Phi(\bb)\|_2 &= \|\int_{\overline{B}} \theta(\tt)\nabla_\bb \phi(\bb-\rho(\bb)\tt) d\tt\|_2=\|\int_{\overline{B}} \theta(\tt)\big (\nabla \phi(\bb-\rho(\bb)\tt) -\langle \nabla \phi(\bb-\rho(\bb)\tt),\tt \rangle \nabla \rho(\bb) \big ) d\tt\|_2\\
&\leq \int_{\overline{B}} \theta(\tt)\|\nabla \phi(\bb-\rho(\bb)\tt\|_2 d\tt +\int_{\overline{B}} \theta(\tt) \|\nabla \phi(\bb-\rho(\bb)\tt)\|_2 \|\tt\|_2 \|\nabla \rho(\bb) \|_2 d\tt
\end{split}
\end{equation*}

By Lagrange's mean value theorem it is enough to show $\|\|\nabla \Phi\|_2\|_\infty\leq \tau+\epsilon$. By Lemma \ref{lem:Lipschitz gradient} for almost all $\tt\in \overline{B}$, $\|\nabla \phi(\bb-\rho(\bb)\tt)\|_2\leq \tau$. Thus  for all $\bb\in \Omega$
\begin{equation*}
\|\nabla \Phi(\bb)\|_2 \leq  \int_{\overline{B}} \theta(\tt)\tau d\tt +\int_{\overline{B}} \theta(\tt) \tau \frac{\epsilon}{\tau}d\tt
=\tau+\epsilon.
\end{equation*}

\end{proof}

\section{Miscellaneous}
\begin{prop}\label{prop:X_nf is Borel}
    Let $(X,\mathcal{X},G)$ be a Borel system. Then the set 
    $$\Xnf:=\{x\in X: \exists g\in G\setminus\{e\},\, \, gx=x\}$$ 
    is a $G$-invariant Borel set.

\end{prop}

\begin{proof}
    If $gx=x$ then for all $h\in G$, $hgx=hx$, thus $\Xnf$ is $G$-invariant. By 
    \cite[Theorem 2.3.2]{BK96} for all $x\in X$, $G_x=\{g\in G|\,\, gx=x\}$ is closed. Let $F(G)$ be the space of all non-empty closed subsets of $G$ equipped with the Fell topology (\cite{fell1962hausdorff}). By \cite[Proposition II.2.3]{auslander1966unitary} the map $s:X\rightarrow F(G)$ given by $s(x)=G_x$, is a Borel map. Note  $\Xnf=s^{-1}(F(G)\setminus\{e\})$. Thus $\Xnf$ is a Borel set.  
\end{proof}

\bibliographystyle{alpha}
\bibliography{flow, universal_bib}

@article {Beb40,
    AUTHOR = {Mikha\u{\i}l V. Bebutov},
     TITLE = {On dynamical systems in the space of continuous functions},
   JOURNAL = {Byull. Moskov. Gos. Univ. Mat.},
  
    VOLUME = {2 (5)},
      YEAR = {1940},
     PAGES = {1--52},
     }

@article {Ebe73,
    AUTHOR = {Ernst Eberlein},
     TITLE = {Einbettung von {S}tr\"{o}mungen in {F}unktionenr\"{a}ume durch
              {E}rzeuger vom endlichen {T}yp},
   JOURNAL = {Z. Wahrscheinlichkeitstheorie und Verw. Gebiete},
  FJOURNAL = {Zeitschrift f\"{u}r Wahrscheinlichkeitstheorie und Verwandte
              Gebiete},
    VOLUME = {27},
      YEAR = {1973},
     PAGES = {277--291},
   MRCLASS = {28A65},
  MRNUMBER = {364600},
MRREVIEWER = {U. Krengel},
       DOI = {10.1007/BF00532824},
       URL = {https://doi.org/10.1007/BF00532824},
}

@article {GJT19,
    AUTHOR = {Yonatan Gutman and Lei Jin and Masaki Tsukamoto},
     TITLE = {A {L}ipschitz refinement of the {B}ebutov-{K}akutani dynamical
              embedding theorem},
   JOURNAL = {J. Dynam. Differential Equations},
  FJOURNAL = {Journal of Dynamics and Differential Equations},
    VOLUME = {31},
      YEAR = {2019},
    NUMBER = {1},
     PAGES = {81--91},
      ISSN = {1040-7294},
   MRCLASS = {37B05},
  MRNUMBER = {3935136},
MRREVIEWER = {Suhua Wang},
       DOI = {10.1007/s10884-018-9653-3},
       URL = {https://doi.org/10.1007/s10884-018-9653-3},
}

@article {Jew69,
    AUTHOR = {Robert I. Jewett},
     TITLE = {The prevalence of uniquely ergodic systems},
   JOURNAL = {J. Math. Mech.},
    VOLUME = {19},
      YEAR = {1970},
     PAGES = {717--729},
   MRCLASS = {28.70 (34.00)},
  MRNUMBER = {0252604},
MRREVIEWER = {J. C. Oxtoby},
}

@article {Kak68,
    AUTHOR = {Shizuo Kakutani},
     TITLE = {A proof of {B}eboutov's theorem},
   JOURNAL = {J. Differential Equations},
  FJOURNAL = {Journal of Differential Equations},
    VOLUME = {4},
      YEAR = {1968},
     PAGES = {194--201},
      ISSN = {0022-0396},
   MRCLASS = {34.65},
  MRNUMBER = {226144},
MRREVIEWER = {Stephen P. Diliberto},
       DOI = {10.1016/0022-0396(68)90036-3},
       URL = {https://doi.org/10.1016/0022-0396(68)90036-3},
}

@inproceedings {Kri72,
    AUTHOR = {Krieger, Wolfgang},
     TITLE = {On unique ergodicity},
 BOOKTITLE = {Proceedings of the {S}ixth {B}erkeley {S}ymposium on
              {M}athematical {S}tatistics and {P}robability ({U}niv.
              {C}alifornia, {B}erkeley, {C}alif., 1970/1971), {V}ol. {II}:
              {P}robability theory},
     PAGES = {327--346},
 PUBLISHER = {Univ. California Press, Berkeley, Calif.},
      YEAR = {1972},
   MRCLASS = {28A65},
  MRNUMBER = {0393402},
MRREVIEWER = {J. C. Oxtoby},
}

@article {EFKKS17,
    AUTHOR = {Eberlein, Ernst and F\"{o}llmer, Hans and Keane, Michael and
              Krengel, Ulrich and Strassen, Volker},
     TITLE = {Konrad {J}acobs (1928--2015)},
   JOURNAL = {Jahresber. Dtsch. Math.-Ver.},
  FJOURNAL = {Jahresbericht der Deutschen Mathematiker-Vereinigung},
    VOLUME = {119},
      YEAR = {2017},
    NUMBER = {3},
     PAGES = {187--199},
      ISSN = {0012-0456},
   MRCLASS = {01A70 (37-03 94-03)},
  MRNUMBER = {3670778},
       DOI = {10.1365/s13291-017-0159-4},
       URL = {https://doi.org/10.1365/s13291-017-0159-4},
}

@incollection {Kat77,
    AUTHOR = {Katok, Anatole B.},
     TITLE = {The special representation theorem for multi-dimensional group
              actions},
 BOOKTITLE = {Dynamical systems, {V}ol. {I}---{W}arsaw},
    SERIES = {Ast\'{e}risque, No. 49},
     PAGES = {117--140},
 PUBLISHER = {Soc. Math. France, Paris},
      YEAR = {1977},
   MRCLASS = {28A65 (58F15)},
  MRNUMBER = {0492181},
MRREVIEWER = {D. Newton},
}

@book {HM66,
    AUTHOR = {Hofmann, Karl H. and Mostert, Paul S.},
     TITLE = {Elements of compact semigroups},
 PUBLISHER = {Charles E. Merrill Books, Inc., Columbus, Ohio},
      YEAR = {1966},
     PAGES = {xiii+384},
   MRCLASS = {22.00},
  MRNUMBER = {0209387},
MRREVIEWER = {R. J. Koch},
}

@article {Gle50,
    AUTHOR = {Gleason, Andrew M.},
     TITLE = {Spaces with a compact {L}ie group of transformations},
   JOURNAL = {Proc. Amer. Math. Soc.},
  FJOURNAL = {Proceedings of the American Mathematical Society},
    VOLUME = {1},
      YEAR = {1950},
     PAGES = {35--43},
      ISSN = {0002-9939},
   MRCLASS = {20.0X},
  MRNUMBER = {33830},
MRREVIEWER = {R. Godement},
       DOI = {10.2307/2032430},
       URL = {https://doi.org/10.2307/2032430},
}

@book {Jawo74,
    AUTHOR = {Jaworski, Allan},
     TITLE = {The {K}akutani-{B}ebutov theorem for groups},
      NOTE = {Ph.D. Thesis, University of Maryland, College Park},
      YEAR = {1974},
     PAGES = {61},
   MRCLASS = {Thesis},
  MRNUMBER = {2624419},
       
}

@book {Ban51,
	AUTHOR = {Banach, Stefan},
	TITLE = {Wst\c{e}p do teorii funkcji rzeczywistych (Polish) [Introduction to the theory of real functions]},
	SERIES = {Monografie Matematyczne. Tom XVII.},
	PUBLISHER = {Polskie Towarzystwo Matematyczne, Warszawa-Wroc\l{}aw},
	YEAR = {1951},
	PAGES = {iv+224},
	MRCLASS = {27.2X},
	MRNUMBER = {0043161},
	MRREVIEWER = {S. Ulam},
}

@book {Kec95,
    AUTHOR = {Kechris, Alexander S.},
     TITLE = {Classical descriptive set theory},
    SERIES = {Graduate Texts in Mathematics},
    VOLUME = {156},
 PUBLISHER = {Springer-Verlag, New York},
      YEAR = {1995},
     PAGES = {xviii+402},
      ISBN = {0-387-94374-9},
   MRCLASS = {03E15 (03-01 03-02 04A15 28A05 54H05 90D44)},
  MRNUMBER = {1321597},
MRREVIEWER = {Jakub Jasi\'{n}ski},
       DOI = {10.1007/978-1-4612-4190-4},
       URL = {https://doi.org/10.1007/978-1-4612-4190-4},
}

@article {Tie15,
    AUTHOR = {Tietze, Heinrich},
     TITLE = {\"{U}ber {F}unktionen, die auf einer abgeschlossenen {M}enge
              stetig sind},
   JOURNAL = {J. Reine Angew. Math.},
  FJOURNAL = {Journal f\"{u}r die Reine und Angewandte Mathematik. [Crelle's
              Journal]},
    VOLUME = {145},
      YEAR = {1915},
     PAGES = {9--14},
      ISSN = {0075-4102},
   MRCLASS = {DML},
  MRNUMBER = {1580909},
       DOI = {10.1515/crll.1915.145.9},
       URL = {https://doi.org/10.1515/crll.1915.145.9},
}

@article {Slu23,
    AUTHOR = {Slutsky, Konstantin},
     TITLE = {Katok's special representation theorem for multidimensional
              {B}orel flows},
   JOURNAL = {Ergodic Theory Dynam. Systems},
  FJOURNAL = {Ergodic Theory and Dynamical Systems},
    VOLUME = {44},
      YEAR = {2024},
    NUMBER = {7},
     PAGES = {1945--1962},
      ISSN = {0143-3857},
   MRCLASS = {37A10 (03E15 37A20)},
  MRNUMBER = {4771304},
       DOI = {10.1017/etds.2023.62},
       URL = {https://doi.org/10.1017/etds.2023.62},
}

@article {Ros85,
    AUTHOR = {Rosenthal, Alain},
    TITLE = {Strictly ergodic models and amenable group actions},
journal={Unpublished manuscript},
    YEAR = {1985},
}

@book {BK96,
    AUTHOR = {Becker, Howard and Kechris, Alexander S.},
     TITLE = {The descriptive set theory of {P}olish group actions},
    SERIES = {London Mathematical Society Lecture Note Series},
    VOLUME = {232},
 PUBLISHER = {Cambridge University Press},
   ADDRESS = {Cambridge},
      YEAR = {1996},
     PAGES = {xii+136},
      ISBN = {0-521-57605-9},
   MRCLASS = {54H05 (04A15 04A20 28A05)},
  MRNUMBER = {1425877 (98d:54068)},
MRREVIEWER = {Klaas Pieter Hart},
       DOI = {10.1017/CBO9780511735264},
       URL = {https://fennec.univ-mlv.fr:443/http/dx.doi.org/10.1017/CBO9780511735264},
}

@article{fell1962hausdorff,
  title={A {H}ausdorff topology for the closed subsets of a locally compact non-{H}ausdorff space},
  author={Fell, James M.G.},
  journal={Proc. Amer. Math. Soc.},
  volume={13},
  number={3},
  pages={472--476},
  year={1962},
  publisher={JSTOR}
}

@book{auslander1966unitary,
  title={Unitary representations of solvable Lie groups},
  author={Auslander, Louis and Moore, Calvin C.},
  number={1-62},
  year={1966},
  publisher={American Mathematical Soc.}
}

@article {Wei25,
    AUTHOR = {Weiss, Benjamin},
     TITLE = {On the {J}ewett-{K}rieger theorem for amenable groups},
   JOURNAL = {Bull. Pol. Acad. Sci. Math.},
  FJOURNAL = {Bulletin of the Polish Academy of Sciences. Mathematics},
    VOLUME = {72},
      YEAR = {2024},
    NUMBER = {2},
     PAGES = {127--130},
      ISSN = {0239-7269},
   MRCLASS = {37A15},
  MRNUMBER = {4935868},
       DOI = {10.4064/ba250109-14-1},
       URL = {https://doi.org/10.4064/ba250109-14-1},
}

@article{struble1974metrics,
  title={Metrics in locally compact groups},
  author={Struble, Raimond A},
  journal={Compositio Math.},
  volume={28},
  number={3},
  pages={217--222},
  year={1974}
}

@article{JZ21,
  title={Minimal model-universal flows for locally compact {P}olish groups},
  author={Jahel, Colin and Zucker, Andy},
  journal={Israel J. Math.},
  volume={244},
  pages={743--758},
  year={2021},
}

@article{hjorth1999sharper,
  title={Sharper changes in topologies},
  author={Hjorth, Greg},
  journal={Proceedings of the American Mathematical Society},
  volume={127},
  number={1},
  pages={271--278},
  year={1999}
}

@book {EvansPDE,
    AUTHOR = {Evans, Lawrence C.},
     TITLE = {Partial differential equations},
    SERIES = {Graduate Studies in Mathematics},
    VOLUME = {19},
   EDITION = {Second},
 PUBLISHER = {American Mathematical Society, Providence, RI},
      YEAR = {2010},
     PAGES = {xxii+749},
      ISBN = {978-0-8218-4974-3},
   MRCLASS = {35-01},
  MRNUMBER = {2597943},
MRREVIEWER = {Diego M. Maldonado},
       DOI = {10.1090/gsm/019},
       URL = {https://doi.org/10.1090/gsm/019},
}

@article{gutman2025lipschitz,
  title={A {L}ipschitz Refinement of the Multidimensional {B}ebutov--{K}akutani Dynamical Embedding Theorem},
  author={Gutman, Yonatan and Huo, Qiang and Tsukamoto, Masaki},
  journal={arXiv preprint arXiv:2510.08706},
  year={2025}
}

@preamble{
   "\def\cprime{$'$} "
}

@article{Kechris1992,
  author    = {Alexander S. Kechris},
  title     = {Countable sections for locally compact group actions},
  journal   = {Ergodic Theory and Dynam. Systems},
  volume    = {12},
  number    = {2},
  pages     = {283--295},
  year      = {1992},
  publisher = {Cambridge University Press},
  doi       = {10.1017/S0143385700006753}
}

@article{FeldmanHahnMoore1979,
  author    = {Jacob Feldman and Peter Hahn and Calvin C. Moore},
  title     = {Orbit structure and countable sections for actions of continuous groups},
  journal   = {Advances in Math.},
  volume    = {28},
  number    = {3},
  pages     = {186--230},
  year      = {1978},
  publisher = {Elsevier},
  doi       = {10.1016/0001-8708(79)90032-5}
}

@book{CastaingValadier1977,
  author    = {Castaing, Charles and Valadier, Michel},
  title     = {Convex Analysis and Measurable Multifunctions},
  series    = {Lecture Notes in Mathematics},
  volume    = {580},
  publisher = {Springer-Verlag},
  year      = {1977},
  doi       = {10.1007/BFb0086562},
  isbn      = {978-3-540-08314-8}
}

@article{czarnecki2006approximation,
  title={Approximation and regularization of {L}ipschitz functions: convergence of the gradients},
  author={Czarnecki, Marc-Olivier and Rifford, Ludovic},
  journal={Trans. Amer. Math. Soc.},
  volume={358},
  number={10},
  pages={4467--4520},
  year={2006}
}

@article{deville2019approximation,
  title={Approximation of {L}ipschitz functions preserving boundary values},
  author={Deville, Robert and Mudarra, Carlos},
  journal={J. Optim. Theory Appl.},
  volume={182},
  pages={885--905},
  year={2019}
  }

@article{gutev2020lipschitz,
  title={Lipschitz extensions and approximations},
  author={Gutev, Valentin},
  journal={J. Math. Anal. Appl.},
  volume={491},
  number={1},
  pages={124242},
  year={2020},
  publisher={Elsevier}
}

@article{Whitney1934AnalyticEO,
  title={Analytic Extensions of Differentiable Functions Defined in Closed Sets},
  author={Hassler Whitney},
  journal={Trans. Amer. Math. Soc.},
  year={1934},
  volume={36},
  pages={63-89},
  url={https://api.semanticscholar.org/CorpusID:120910041}
}

@article{McSHANE1934ExtensionOR,
  title={Extension of range of functions},
  author={Edward James McShane},
  journal={Bull. Amer. Math. Soc.},
  year={1934},
  volume={40},
  pages={837-842},
  url={https://api.semanticscholar.org/CorpusID:38462037}
}

@article {DE74,
    AUTHOR = {Manfred Denker and Ernst Eberlein},
     TITLE = {Ergodic flows are strictly ergodic},
   JOURNAL = {Advances in Math.},
  FJOURNAL = {Advances in Mathematics},
    VOLUME = {13},
      YEAR = {1974},
     PAGES = {437--473},
      ISSN = {0001-8708},
   MRCLASS = {28A65},
  MRNUMBER = {352403},
MRREVIEWER = {V. Drobot},
       DOI = {10.1016/0001-8708(74)90075-9},
       URL = {https://doi.org/10.1016/0001-8708(74)90075-9},
}

@article{GQS18,
    AUTHOR = {Gutman, Yonatan and Qiao, Yixiao and Szab\'{o}, G\'{a}bor},
     TITLE = {The embedding problem in topological dynamics and {T}akens' theorem},
   JOURNAL = {Nonlinearity},
  FJOURNAL = {Nonlinearity},
    VOLUME = {31},
      YEAR = {2018},
    NUMBER = {2},
     PAGES = {597--620},
      ISSN = {0951-7715},
   MRCLASS = {37B50 (54H20)},
  MRNUMBER = {3755880},
       DOI = {10.1088/1361-6544/aa9464},
       URL = {https://doi.org/10.1088/1361-6544/aa9464},
}

@article {cornet1999smooth,
    AUTHOR = {Cornet, Bernard and Czarnecki, Marc-Olivier},
     TITLE = {Smooth representations of epi-{L}ipschitzian subsets of
              {$\bold R^n$}},
   JOURNAL = {Nonlinear Anal. Theory Methods Appl.},
  FJOURNAL = {Nonlinear Analysis. Theory, Methods \& Applications. An
              International Multidisciplinary Journal},
    VOLUME = {37},
      YEAR = {1999},
    NUMBER = {2},
     PAGES = {139--160},
      ISSN = {0362-546X},
   MRCLASS = {49J53 (90C30)},
  MRNUMBER = {1689740},
MRREVIEWER = {Adam B. Levy},
       DOI = {10.1016/S0362-546X(98)00033-9},
       URL = {https://doi.org/10.1016/S0362-546X(98)00033-9},
}

@article {G,
    AUTHOR = {Gromov, Misha},
     TITLE = {Topological invariants of dynamical systems and spaces of
              holomorphic maps. {I}},
   JOURNAL = {Math. Phys. Anal. Geom.},
  FJOURNAL = {Mathematical Physics, Analysis and Geometry. An International
              Journal Devoted to the Theory and Applications of Analysis and
              Geometry to Physics},
    VOLUME = {2},
      YEAR = {1999},
    NUMBER = {4},
     PAGES = {323--415},
      ISSN = {1385-0172},
     CODEN = {MPAGFO},
   MRCLASS = {37B99 (32H02 53C23 58E20)},
  MRNUMBER = {MR1742309 (2001j:37037)},
MRREVIEWER = {Boris Hasselblatt},
}

@book {A,
    AUTHOR = {Auslander, Joseph},
     TITLE = {Minimal flows and their extensions},
    SERIES = {North-Holland Mathematics Studies},
    VOLUME = {153},
      NOTE = {Notas de Matem\'atica [Mathematical Notes], 122},
 PUBLISHER = {North-Holland Publishing Co.},
   ADDRESS = {Amsterdam},
      YEAR = {1988},
     PAGES = {xii+265},
      ISBN = {0-444-70453-1},
   MRCLASS = {54H20},
  MRNUMBER = {MR956049 (89m:54050)},
MRREVIEWER = {M. Rees},
}

@book {N,
    AUTHOR = {Nagata, Jun-iti},
     TITLE = {Modern dimension theory},
    SERIES = {Sigma Series in Pure Mathematics},
    VOLUME = {2},
   EDITION = {Revised},
 PUBLISHER = {Heldermann Verlag},
   ADDRESS = {Berlin},
      YEAR = {1983},
     PAGES = {ix+284},
      ISBN = {3-88538-002-1},
   MRCLASS = {54F45},
  MRNUMBER = {MR715431 (84h:54033)},
}

@article{gutman2020embedding,
  title={Embedding minimal dynamical systems into {H}ilbert cubes},
  author={Gutman, Yonatan and Tsukamoto, Masaki},
  journal={Invent. Math.},
  VOLUME = {221},
  PAGES = {113--166},
  year={2020}
}

@article{Gut15Jaworski,
author = {Gutman, Yonatan},
title = {Mean dimension and {J}aworski-type theorems},
volume = {111},
number = {4},
pages = {831-850},
year = {2015},
doi = {10.1112/plms/pdv043},
URL = {http://plms.oxfordjournals.org/content/111/4/831.abstract},
eprint = {http://plms.oxfordjournals.org/content/111/4/831.full.pdf+html},
journal = {Proc. Lond. Math. Soc.}
}

@article {W85,
    AUTHOR = {Weiss, Benjamin},
     TITLE = {Strictly ergodic models for dynamical systems},
   JOURNAL = {Bull. Amer. Math. Soc. (N.S.)},
  FJOURNAL = {American Mathematical Society. Bulletin. New Series},
    VOLUME = {13},
      YEAR = {1985},
    NUMBER = {2},
     PAGES = {143--146},
      ISSN = {0273-0979},
     CODEN = {BAMOAD},
   MRCLASS = {28D05},
  MRNUMBER = {799798 (87c:28026)},
MRREVIEWER = {F. M. Dekking},
       DOI = {10.1090/S0273-0979-1985-15399-6},
       URL = {http://dx.doi.org/10.1090/S0273-0979-1985-15399-6},
}

@article{dranishnikov2025free,
  title={A free $\mathbb{Z}$-action by isometries on a compact metric space which is not embeddable into a cubical shift},
  author={Dranishnikov, Alexander and Levin, Michael},
  journal={arXiv preprint arXiv:2508.14628},
  year={2025}
}

@article {LT12,
  AUTHOR = {Lindenstrauss, Elon and Tsukamoto, Masaki},
  TITLE = {Mean Dimension and an embedding problem: an example},
   JOURNAL = {Israel J. Math.},
   FJOURNAL = {Israel Journal of Mathematics},
    VOLUME = {199},
     PAGES = {573--584},
  YEAR = {2014},
}

@book {G03,
    AUTHOR = {Glasner, Eli},
     TITLE = {Ergodic theory via joinings},
    SERIES = {Mathematical Surveys and Monographs},
    VOLUME = {101},
 PUBLISHER = {American Mathematical Society},
   ADDRESS = {Providence, RI},
      YEAR = {2003},
     PAGES = {xii+384},
      ISBN = {0-8218-3372-3},
   MRCLASS = {37A15 (28Dxx 37A25 37A35 37A45 37B99 54H20)},
  MRNUMBER = {1958753 (2004c:37011)},
MRREVIEWER = {Andr{\'e}s del Junco},
}

\noindent Yonatan Gutman, Institute of Mathematics, Polish Academy of Sciences, ul. \'{S}niadeckich 8, Warsaw, 00-656, Poland.

\noindent\textit{Email address}: \texttt{gutman@impan.pl}

\noindent Qiang Huo, School of Mathematical Sciences, University of Science and Technology of China, Hefei, Anhui, 230026, P. R. China.

\noindent\textit{Email address}: \texttt{qianghuo@ustc.edu.cn}

\end{document}